\documentclass[11pt,a4paper]{amsart}

\usepackage[a4paper,hmargin=3cm,vmargin=3cm]{geometry}
\usepackage{amsmath,amsthm,amsfonts,amssymb,amscd,amstext}
\usepackage[latin1]{inputenc}
\usepackage[english]{babel}
\usepackage{color}
\usepackage{graphicx}
\usepackage{hyperref} 

\AtBeginDocument{%
  \setlength{\abovedisplayskip}{0.8em plus .15em minus .15em}
  \setlength{\belowdisplayskip}{0.8em plus .15em minus .15em}
  \setlength{\abovedisplayshortskip}{0.6em plus .1em}
  \setlength{\belowdisplayshortskip}{0.6em plus .1em}
  \setlength{\parskip}{0.3em plus .1em minus .05em}
}
\makeatletter
\def\and{\unskip\space and\space}
\makeatother

\def\R{\mathbb{R}}
\def\B{\mathbb{B}}
\def\N{\mathbb{N}}
\def\Z{\mathbb{Z}}
\def\C{\mathbb{C}}

\def\D{\mathbb{D}}

\newcommand{\ben}{\begin{enumerate}}
\newcommand{\bit}{\begin{itemize}}
\newcommand{\een}{\end{enumerate}}
\newcommand{\eit}{\end{itemize}}

\newcommand{\ed}{\end{document}}

\def\cU{\mathcal{U}}
\def\cP{\mathcal{P}}

\def\cC{\mathcal{C}}

\def\cB{\mathcal{B}}
\def\cW{\mathcal{W}}

\def\cV{\mathcal{V}}
\def\cJ{\mathcal{J}}

\def\cH{\mathcal{H}}

\def\cM{\mathcal{M}}

\def\cG{\mathcal{G}}
\def\cF{\mathcal{F}}

\def\cK{\mathcal{K}}
\def\cO{\mathcal{O}}

\def\cI{\mathcal{I}}

\let\8=\infty \let\0=\emptyset  
\let\hat=\widehat

\let\landa=\lambda
\let\alfa=\alpha

\let\parc=\partial

\def\ep{\varepsilon}
\def\vart{\vartheta}
\def\landa{\lambda}

\def\flecha{\rightarrow}

\def\cte.{\mathop{\rm cte.}\nolimits}

\def\det{\mathop{\rm det}\nolimits}

\def\Re{\mathop{\rm Re }\nolimits}

\def\cosh{\mathop{\rm cosh }\nolimits}

\def\N{\mathbb{N}}

\def\B{\mathbb{B}}

\def\R{\mathbb{R}}
\def\Z{\mathbb{Z}}
\def\C{\mathbb{C}}
\def\I{\mathbb{I}}

\def\D{\mathbb{D}}

\newcommand{\V}{\mathcal{V}}

\numberwithin{equation}{section} 

\newtheorem{theorem}{Theorem}[section]
\newtheorem{lemma}[theorem]{Lemma}
\newtheorem{proposition}[theorem]{Proposition}

\newtheorem{remark}[theorem]{Remark}
\newtheorem{corollary}[theorem]{Corollary}
\newtheorem{definition}[theorem]{Definition}

\newtheorem{claim}[theorem]{Claim}

\begin{document}

\thanks{Research supported by Project PID2024-160586NB-I00 funded by MCIN/AEI/10.13039/501100011033}

\title{All Serrin bands are unduloidal}

\author{Laurent Hauswirth}

\address{Laboratoire d'Analyse et de Matématiques Apliquées, Université Gustave Eiffel,
Paris, France.}
\email{laurent.hauswirth@u-pem.fr} 

\author{Pablo Mira}

\address{Departamento de Matematica Aplicada y Estad\' \i stica, Universidad Politecnica 
de Cartagena, Spain} \email{pablo.mira@upct.es}


\maketitle
\begin{abstract}
A Serrin domain is a smooth domain $\Omega\subset \R^n$ where Serrin's overdetermined problem for $\Delta u+2=0$ admits a positive solution $u$. In the plane, Fall-Minlend-Weth and Dávila-del Pino-Musso-Wheeler constructed periodic Serrin bands close to a flat strip and a singular chain of tangent disks, respectively. More recently, Cerezo-Fernández-Mira constructed a real-analytic $1$-parameter family of unduloidal Serrin bands interpolating between these two limiting geometries. In this work, we prove that every Serrin band in $\mathbb R^2$ belongs to this canonical unduloidal family.

\end{abstract}

\section{Introduction}
\subsection{Statement of the main theorem}\label{sec:1.1}
A fundamental question in the theory of overdetermined elliptic problems is to classify the smooth domains 
$\Omega\subset \R^n$ that support a positive solution $u>0$ to 
\begin{equation}\label{serrinp0}\def\arraystretch{1.8}\left\{\begin{array}{lll} \Delta u +f(u)=0 & \text{ in } & \Omega, \\
u=0, \hspace{0.5cm} \displaystyle \frac{\parc u}{\parc \nu} = -c  & \text{ on } & \parc \Omega, 
\end{array} \right.
\end{equation} 
for some $c\geq 0$. Here, $\nu$ is the outward unit normal of $\parc \Omega$, and $f$ is sufficiently regular. When $\Omega$ is bounded, Serrin's famous result \cite{Se} proves that $\Omega$ is a ball. For unbounded domains, the classical approach (Berestycki, Cafferelli and Nirenberg \cite{BCN}, Reichel \cite{Re}) seeks rigidity of symmetric domains such as half-spaces or complements of balls; see also \cite{BF,FV,RS,RRS1,RSW2,Sir,WW}. 
In contrast, recent advances in the theory have shown the existence of many non-symmetric \emph{exceptional} domains $\Omega\subset \R^n$ where \eqref{serrinp0} can be solved. Starting with Sicbaldi \cite{Sic}, several key constructions in this direction, usually obtained via perturbative methods, are  \cite{ABM,CF,CP,CS,CSS,DZ,DDMW,DPW,EFRS,FMW1,FMW3,FMW2,KS,LWW,RRS2,Ru,RSW,Sic,Sic2,SS,Wh}.

Therefore, in order to achieve a description of the solutions to \eqref{serrinp0}, one needs to obtain global uniqueness or structure theorems for these non-standard families of solutions. Not much is currently understood in that regard. Ideally, one would like to prove that, when such exceptional domains exist, they are organized into global moduli spaces and can be classified just based on their topology. Theorem \ref{th:main} below provides such a topological classification for solutions to the torsion equation $\Delta u +2=0$ in $\R^2$.

Problem \eqref{serrinp0} for $\Delta u +2=0$ was already singled out by Serrin \cite{Se} as a case of particular interest. Indeed, this equation appears naturally in different problems of fluid mechanics and elasticity, and has also been proved of interest in architectural design \cite{PL}. Motivated by this, we will say that a smooth domain $\Omega\subset \R^n$ is a \emph{Serrin domain} if it supports a positive solution to
\begin{equation}\label{serrinprob}\def\arraystretch{1.8}\left\{\begin{array}{lll} \Delta u +2=0 & \text{ in } & \Omega, \\
u=0, \hspace{0.5cm} \displaystyle \frac{\parc u}{\parc \nu} = -c <0 & \text{ on } & \parc \Omega.
\end{array} \right.
\end{equation} 
There are very few known Serrin domains in $\R^n$, beyond the trivial ones of balls $\B^n$ and generalized cylinders $\B^{k}\times \R^{n-k}$, with $1\leq k\leq n-1$. The first non-trivial Serrin domains were constructed by Fall, Minlend and Weth \cite{FMW1}, by bifurcation from adequate cylinders $\B^{k}\times \R^{n-k}$. For $n=2$, these examples describe periodic Serrin bands sufficiently close to a flat band $\{(x_1,x_2): |x_2|\leq r\}$. In \cite{DDMW}, Dávila, Del Pino, Musso and Wheeler  constructed periodic Serrin bands in the plane, by desingularization of a chain of tangent disks along an axis. In \cite{CFM}, Cerezo, Fernández and Mira recently constructed a global real analytic $1$-dimensional family 
\begin{equation}\label{ffamily}
\cF=\{\Omega_\tau : \tau\in (0,1]\}
\end{equation} 
of periodic Serrin bands, which interpolates between a flat band ($\tau=1$) and a chain of tangent disks ($\tau=0$). We call them \emph{unduloidal Serrin bands}, since their geometric shape and moduli space structure resembles the usual rotational constant mean curvature (CMC) surfaces. See Figure \ref{fig1}.

The construction in \cite{CFM} of the unduloidal family \eqref{ffamily} uses the geometric ansatz of a foliation by \emph{planar} capillary curves. This means that each capillary curve $\gamma\subset \Omega_\tau$ of the foliation has the property that the graph of the solution $u$ along $\gamma$ is contained in a plane of $\R^3$ given by an affine function $P$, and $|\nabla (u-P)|$ is constant along that curve. As a consequence of this construction, all unduloidal bands can be explicitly described in terms of elliptic functions. This foliation is inspired by geometry, see \cite{CFM2,CFM3,FHM}.

\begin{figure}[h]\label{fig1}
  \centering
 \includegraphics[width=0.86\textwidth]{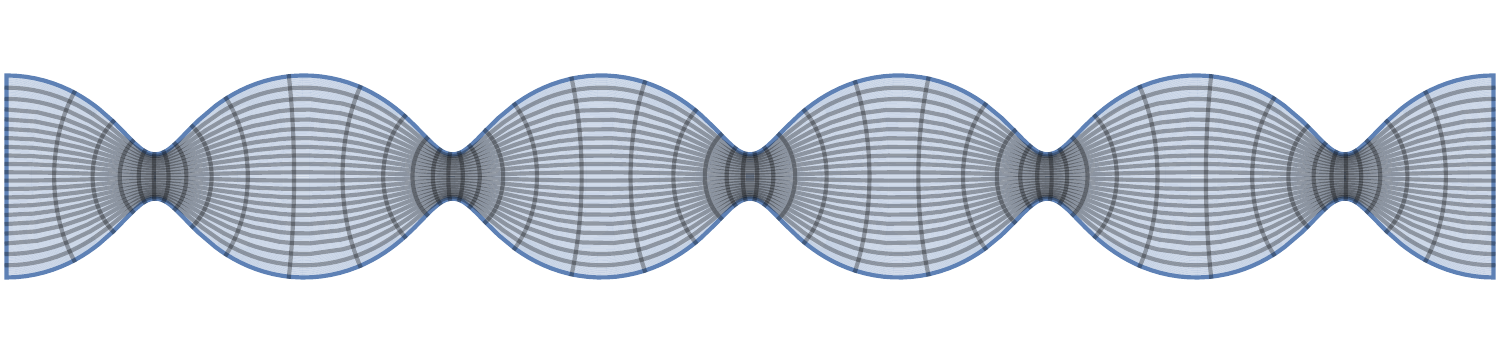} \\[0.6cm] \includegraphics[width=0.72\textwidth]{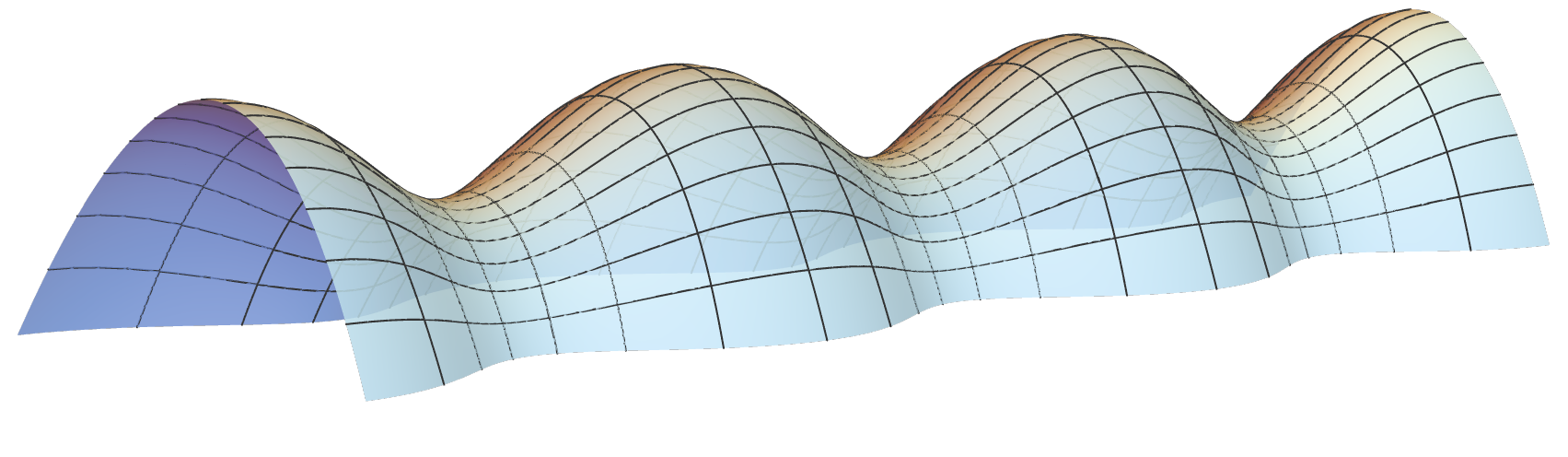}  \\    
  \caption{Top: A member $\Omega_\tau$ of the unduloidal family $\cF$ of Serrin bands. Bottom: The graph of the Serrin solution $u$ on the unduloidal band $\Omega_\tau$.\label{figbandas}} 
  \end{figure}

Let $\Omega\subset \R^2$ be an unbounded Serrin domain. Then, by Ros-Sicbaldi theory \cite{RS}, if $\Omega$ has finite connectivity and at most two \emph{unbounded} boundary components, it must be up to Euclidean symmetry a symmetric bigraph of the form
\begin{equation}\label{bigraph}
\Omega=\{(x_1,x_2) : |x_2| < f(x_1)\},
\end{equation} 
with $f:\R\flecha \R$ a positive, bounded, real analytic function. We call such $\Omega$ a \emph{Serrin band}.

Our main result is a complete classification of Serrin bands:

\begin{theorem}\label{th:main}
Any Serrin band $\Omega\subset\R^2$ is unduloidal, i.e., up to Euclidean similarity it is an element of the family $\cF:=\{\Omega_\tau : \tau\in (0,1]\}$ in \eqref{ffamily}. \end{theorem}

We remark that Theorem \ref{th:main} makes no periodicity or symmetry assumptions on $\Omega$, and proves in particular that the Serrin bands in \cite{FMW1,DDMW} belong to the unduloidal family $\cF$.

Theorem \ref{th:main} is the overdetermined version of an important geometric theorem by Korevaar, Kusner and Solomon \cite{KKS}, according to which any properly embedded CMC surface of finite genus and two ends in $\R^3$ is a rotational unduloid. Theorem \ref{th:main} can also be seen as an overdetermined version of the celebrated uniqueness theorem for Riemann's minimal examples by Meeks-Pérez-Ros \cite{MPR}. The parallelism here comes from the fact that Riemann's examples are foliated by circles and lines in parallel planes, which is a similar foliation structure to that of the Serrin unduloidal bands. So, in both Theorem \ref{th:main} and \cite{MPR}, a foliated family of examples exhausts an entire topological class. We will explain in Section \ref{sketch} how the connection of Theorem \ref{th:main} and the uniqueness result in \cite{MPR} is deeper than some formal analogy.

\subsection{Discussion of the main result}\label{sec:1.2}
Theorem \ref{th:main} gives a path to understand solutions to \eqref{serrinp0} beyond the \emph{rigidity versus flexibility} level discussed above. For instance, it suggests that the unduloidal bifurcation domains for $\Delta u +\landa u=0$ of Sicbaldi \cite{Sic,SS} might have a global moduli space structure, and be unique within its topological type. In a similar way, it motivates seeking a global geometric structure and classification of all Schiffer domains of prescribed simple topology (disks, annuli) in $\R^2$, after the outstanding counterexamples to the Schiffer and Berenstein conjectures recently obtained in \cite{CP,CS,CSS}; see also \cite{EFRS,Wh}. Our proof does not apply to $\Delta u+\landa u=0$, but suggests a blueprint for such topological classification results.

There is one precedent to Theorem \ref{th:main} regarding topological uniqueness of exceptional domains: Traizet's classification \cite{Tr} of all solutions to \eqref{serrinp0} for $\Delta u=0$ in the plane with a finite number of boundary components. Traizet's result relies on some deep theorems of minimal surface theory, and includes the exceptional \emph{hairpin domain} of \cite{HHP}. Theorem \ref{th:main} does not depend directly on surface theory results, but it strengthens the connection between solutions to \eqref{serrinp0} and the geometry of embedded constant mean curvature hypersurfaces. For instance, the uniqueness of Theorem \ref{th:main} and the flux formulas used in its proof give a path to describe the general shape of all Serrin domains of finite topological type  in $\R^2$, in the spirit of \cite{KKS}. We remark that all unbounded Serrin domains are weakly unstable, see \cite{EK}.

Theorem \ref{th:main} also admits a natural interpretation in fluid dynamics. First, recall that solutions to the torsion equation $\Delta u+2=0$ represent stream functions $u$ of steady Euler flows with constant vorticity. In addition, the Serrin boundary conditions in \eqref{serrinprob} imply that $\parc \Omega$ consists of streamlines along which the fluid speed is constant; see e.g. \cite{EFRS}. By Bernoulli's law, the pressure is also constant on $\partial\Omega$; thus, in the absence of gravity and surface tension, the components of $\parc \Omega$ can be regarded as free interfaces. From this perspective, Theorem \ref{th:main} gives a complete classification of such steady free-boundary flows of constant vorticity, when $\parc \Omega$ has two connected components and the stream function is positive in the fluid region and vanishes on its free boundary. Again, no symmetry or periodicity is imposed on $\Omega$.

In the presence of gravity or surface tension, the resulting free boundary conditions are no longer of Serrin type. However, they appear naturally in the theory of water waves of constant vorticity, see e.g. \cite{CV,CSV,DDMW,HW,WWe}. This is a deeply studied topic. Our results here suggest the possibility of using complex analysis and integrable systems to study the uniqueness and moduli space structure of suitable classes of water waves of constant vorticity.

\subsection{Sketch of the proof}\label{sketch}
One noteworthy difference between the Serrin problem and CMC theory is that the unduloidal bands of the family $\cF$ do not come from a reduction by ambient symmetries of the overdetermined problem \eqref{serrinprob}, like rotational CMC surfaces do. Therefore, methods such as moving planes do not seem sufficient for uniqueness here, since any possible proof strategy should be able to detect these unduloidal bands. 

Our approach is inspired by Meeks-Pérez-Ros' uniqueness proof \cite{MPR}, and uses the theory developed in \cite{CFM} for the study of Serrin domains. The main idea is to show that any Serrin band arises, in an adequate conformal parametrization, from an algebro-geometric potential of the mKdV hierarchy with spectral genus $\leq 1$. Proving this is enough, since Serrin bands with spectral genus $\leq 1$ are foliated by capillary curves (\cite{CFM}), and can actually be shown to be elements of the unduloidal family of Theorem \ref{th:main}. As in \cite{MPR}, we will use an isospectral deformation to force any Serrin candidate into uniqueness. However, the holomorphic KdV deformation in \cite{MPR} does not preserve Serrin boundary conditions. To overcome this, we will develop a new approach to uniqueness of overdetermined problems via boundary integrability.

The proof will go through the following steps:

{\bf \underline{Section \ref{sec:apriori}:}} We give a priori estimates for any Serrin band in $\R^2$, showing that their boundary curves have bounded curvature and stay at a positive distance from their central symmetry line. This is done using fundamental ideas of Ros, Ruiz and Sicbaldi \cite{RS,RRS1,RRS2} together with some geometric arguments for curvature estimates for minimal and CMC surfaces using flux. A key element of these a priori estimates is Traizet's characterization of the hairpin domain in the plane \cite{Tr}, which allows us to classify the blow-up limits of Serrin bands and prove the non-collapsing property of the necks by a flux argument.

{\bf \underline{Section \ref{sec:confo}:}} The estimates in Section \ref{sec:apriori} are used to prove that any Serrin band $\Omega\subset \R^2$ can be parametrized by a conformal global coordinate $g:\overline\cU\equiv \R\times [-y^*,y^*]\flecha \Omega$ for which its \emph{Hopf differential} $\mathfrak{q}$ is constant. These parameters will allow us to describe the geometry of the Serrin band in terms of the harmonic function $\omega(x,y)$ on $\cU$ given by $e^{2\omega}=|g'|^2$. More importantly, it places any arbitrary Serrin band in the analytic context used in \cite{CFM}, that was in principle only available for \emph{periodic} Serrin bands. We provide then a conformal representation for Serrin bands with respect to this complex parameter, and describe formulas that control the Serrin boundary conditions.

{\bf \underline{Section \ref{sec:kdv}:}} Here we prove that if $\Omega$ is any Serrin band, then the holomorphic function $\omega_z:\cU\flecha \C$ coming from the canonical conformal parameters of Section \ref{sec:confo} is an algebro-geometric potential of the mKdV hierarchy. This was proved in \cite{CFM} for \emph{periodic} Serrin bands. The idea is that the hierarchy produces a sequence $\{J_n\}_n$ of bounded harmonic function (Jacobi fields) on $\Omega$ that satisfy the same Robin boundary conditions. But the space of such Jacobi fields can be shown to be finite-dimensional. Hence, some element of the hierarchy is a linear combination of the previous ones, and this means $\omega_z$ is algebro-geometric. The importance of this result is that it allows to describe all Serrin bands by algebraic (yet complicated) conditions. In particular, any Serrin band has an associated real \emph{spectral polynomial} $a(\landa)$ of degree $2n$. The number $n$ is the \emph{spectral genus}, and detects the algebraic and geometric complexity of the band. Flat bands correspond to $n=0$, and non-flat unduloidal bands to $n=1$. In Section \ref{sec:kdv} we also describe algebraically the Serrin boundary conditions, and prove that all roots of the spectral polynomial $a(\landa)$ are negative.

{\bf \underline{Section \ref{sec:uniqueness}:}} In this section we prove Theorem \ref{th:main} under the additional hypothesis that the Serrin band has a symmetry line orthogonal to its central symmetry axis. If we view $\Omega$ as in \eqref{bigraph}, this means that $f$ is an even function. The proof is an estimate on the number of negative zeros that the spectral polynomial $a(\landa)$ can have. We show that under this symmetry assumption, and if the algebraic mKdV relation defining $\Omega$ is the \emph{minimal one}, then $a(\landa)$ has degree at most two. Hence $n=1$ and $\Omega$ must be in the unduloidal family.

{\bf \underline{Section \ref{sec:isospectral}:}} After the uniqueness result in Section \ref{sec:uniqueness}, the fundamental idea is to consider an arbitrary Serrin band $\Omega_0$ and show that one can deform it into another one, $\Omega$, that has this additional vertical symmetry line. In order to do this, in Section \ref{sec:isospectral} we apply the well-established isospectral deformation theory for the mKdV equation to our Serrin domains. The complex meromorphic mKdV flows are Hamiltonian commuting flows. In our situation, we restrict to the \emph{real} flows generated by the hierarchic harmonic Jacobi fields $\{J_n\}_n$. We prove such fields are actually integrable, and even more, they preserve the property of being a Serrin band, with the same Neumann angle and conformal strip width as the original one. Moreover, we show that the space of all such \emph{isospectral} Serrin bands is compact, using again the curvature estimates of Section \ref{sec:apriori}.

{\bf \underline{Section \ref{sec:mainth}:}} Once here, we are in the position to complete the proof of Theorem \ref{th:main}. Consider an arbitrary Serrin band $\Omega_0$. First, we prove using the isospectral deformation theory and the compactness result of Section \ref{sec:isospectral} that there exists a Serrin band in the closure of the isospectral orbit of $\Omega_0$ that has the desired vertical symmetry line. By Section \ref{sec:uniqueness}, $\Omega$ is unduloidal. Finally, we show that this can only happen if $\Omega_0$ itself is also unduloidal. This completes the proof.

\subsection*{AI declaration:} The authors conceived the proof outlined above as part of a multi-year research work, and developed all key mathematical ideas and arguments in it without AI assistance. ChatGPT was used at a final stage to simplify the algebraic part of the proof and discuss some technical points, always within our original strategy. The only exception is the technical Lemma \ref{th:appendix} in the Appendix, whose proof was obtained directly from ChatGPT, and rewritten by the authors. The paper contains no AI writing.

\section{A priori estimates and compactness}\label{sec:apriori}

\subsection{Capillary flux}
Solutions to the overdetermined problem 
\begin{equation}\label{selin}
\Delta u -2H=0 \hspace{0.2cm} \text{in $\Omega$}, \hspace{0.5cm} u=0, \hspace{0.3cm} \frac{\parc u}{\parc \nu}=-c, \hspace{0.2cm} \text{on $\parc \Omega$}
\end{equation}
for some $H\in \R$, $c>0$ represent stationary critical points of the energy functional (see \cite{WW})
$$\mathcal{E}(\Omega,u):= \int_\Omega \left( \frac{1}{2} |Du|^2 +2H u +\frac{c^2}{2}\right) dx.$$
Viewing the integrand as a Lagrangian $L(u,\nabla u)$, we arrive at the usual stress-energy tensor (see e.g. \cite{Al})
$T:\Omega\flecha \cM_{2}(\R)$ given by $$T_{ij}= u_{x_i} u_{x_j} -\left( \frac{1}{2} |Du|^2 +2H u +\frac{c^2}{2}\right)\delta_{ij},$$ which satisfies ${\rm div} \,T= (\Delta u-2H) Du=0.$

We now consider an unbounded \emph{band} domain $\Omega\subset \R^2$, with $\overline{\Omega}$ diffeomorphic to $[0,1]\times \R$, where \eqref{selin} can be solved. Let  $\gamma$ be a transverse section, i.e., a smooth compact arc joining the two boundary components of $\Omega$. So, $\gamma$ divides $\Omega$ into two regions. We select one of these two \emph{ends} of $\Omega$ as the positive one, so $\gamma$ has an orientation.

\begin{definition}
The \emph{capillary flux} of $\Omega$ is defined as 
\begin{equation}\label{defiflux}
\cF(\Omega)=\int_\gamma T {\bf n}_\gamma \, \in \R^2,
\end{equation} where ${\bf n}_\gamma$ is the unit normal of $\gamma$ that points towards the \emph{positive} end of $\Omega$.
\end{definition}

We need to check that this is a good definition, i.e., it does not depend on the chosen section. Let $\gamma_1,\gamma_2$ be two such sections, and assume they are disjoint. Let $D$ denote the domain of $\Omega$ bounded by $\gamma_1,\gamma_2$. Then, $$0= \int_D {\rm div} \,T = \int_{\parc D}  T {\bf n}_D,$$ where ${\bf n}_D$ is the exterior unit normal of $\parc D$. Along the capillary boundary arcs of $\parc D\cap \parc \Omega$, ${\bf n}_D$ is equal to the exterior unit normal $\nu$ of $\parc\Omega$, and $Du= -c \nu$. So, using also that $u=0$ along $\parc \Omega$, we have $$T {\bf n}_D =T \nu= c^2 \nu -\left( \frac{c^2}{2} +2 H u +\frac{c^2}{2}\right)\nu =0 \hspace{0.5cm} \text{along $\parc D\cap \parc \Omega.$}$$ From here, we deduce directly the desired independence.

When $u>0$ in $\Omega$, it follows from \cite{RS,Tr} that $\Omega$ is, up to isometry of $\R^2$, a symmetric bigraph \eqref{bigraph}. The flux in this case is of the form $$(\cF(\Omega),0)$$ for some number $\cF(\Omega)$, and can be computed by means of \eqref{defiflux} taking $\gamma$ to be a segment $x_2\mapsto (0,x_2)$, where we have chosen $x_1=0$ just for simplicity. In this situation, 
\begin{equation}\label{fluxsi}
\cF(\Omega)=-c^2 f(0)+\int_{0}^{f(0)} \left( u_{x_1}^2-u_{x_2}^2- 4H u\right)(0,x_2) \, dx_2.
\end{equation}

\subsection{Harmonic functions: the hairpin domain} The hairpin domain was found by Hauswirth, Helein and Pacard in \cite{HHP}. It is given by $$\Omega_0:=\{(x_1,x_2): |x_2|< \cosh(x_1) +\pi/2\},$$ and has the fundamental property of supporting a positive solution to the overdetermined problem
 \begin{equation}\label{hapi}
 \Delta u =0 \hspace{0.3cm} \text{in $\Omega_0$}, \hspace{0.5cm} u=0, \ \frac{\parc u}{\parc \nu} =-c \hspace{0.3cm} \text{on $\parc \Omega_0$},
 \end{equation} 
 for some $c>0$. The domain $\Omega_0$ is parametrized by the conformal mapping $$g(z)=z + \sinh(z): B_0\flecha \Omega_0, \hspace{0.5cm} B_0:=\{z\in \C : |{\rm Im}(z)|<\pi/2\}.$$ The corresponding solution $u$ to \eqref{hapi} is then $u=\mathfrak{v}\circ g^{-1}$, where $\mathfrak{v}(z):= c\,{\rm Re} \cosh(z)$.

To compute the flux of the hairpin $\Omega_0$ we select the transversal segment $\Omega_0\cap \{x_1=0\}$, and note that $\Omega_0$ is symmetric with respect to $x_1=0$. So, 
\begin{equation}\label{intehair1}
\cF(\Omega_0)= -c^2 f(0) +\int_0^{f(0)} (u_{x_1}^2-u_{x_2}^2) (0,x_2)\, dx_2,
\end{equation} 
with $f(0)=1+\pi/2$. Since $u(x_1,x_2)=u(-x_1,x_2)$ we have $u_{x_1}(0,x_2)=0$. We now use the conformal parameter $z=x+iy$. Then the arc $\{x_1=0\}\cap \Omega$ corresponds to $\{g(iy):-\pi/2<y<\pi/2\}$. Since $g(iy)=i(y+\sin y)$, we see that the integral in \eqref{intehair1} is $$-c^2\int_0^{\pi/2} \frac{\sin^2 t}{(1+\cos t)^2} (1+\cos t) dt =-c^2\int_0^{\pi/2}(1-\cos t) dt= c^2(1-\pi/2).$$ Thus 
\begin{equation}\label{herflux}
\cF(\Omega_0)=-c^2\pi.
\end{equation}
\subsection{A curvature estimate}

In this section we consider a Serrin band 
\begin{equation}\label{bandi}
\Omega:=\{(x_1,x_2): |x_2|< f(x_1)\}\subset \R^2,
\end{equation} 
with $f:\R\flecha (0,\8)$ a smooth positive bounded function. We will let $u\in C^{\8}(\Omega)$ denote its associated solution to $\Delta u+2=0$ with Serrin overdetermined boundary conditions:
\begin{equation}\label{serrinproban}\def\arraystretch{1.8}\left\{\begin{array}{lll} \Delta u +2=0 & \text{ in } & \Omega, \\
u=0, \hspace{0.5cm} \displaystyle \frac{\parc u}{\parc \nu} = -c  & \text{ on } & \parc \Omega, 
\end{array} \right.
\end{equation}
where $u>0$ in $\Omega$ and $c>0$ is constant. Our objective is to prove:

\begin{theorem}\label{th:estc2}
In the above conditions, for any $\alfa\in (0,1)$, we have $$||u||_{C^{2,\alfa}(\overline{\Omega})} <\8.$$
\end{theorem}
The proof also yields uniform $C^{2,\alpha}$ boundary charts. By bootstrapping arguments, we obtain then $
||u||_{C^{k,\alpha}(\overline{\Omega})}<\infty
$
for every $k\geq 2$ and $\alfa\in (0,1)$. We will first prove $C^0$ and $C^1$ estimates.

\subsection{Height and gradient estimates}

\begin{lemma}\label{prop:height}
$u$ is bounded. Specifically, 
\begin{equation}\label{c0est}
||u||_{L^{\8}(\Omega)}\leq \frac{c^2}{2}+ ({\rm sup}_\R \,f)^2 <\8.
\end{equation}
\end{lemma}
\begin{proof} First we note that, by Serrin's moving planes method, $u_{x_2}<0$ for any $(x_1,x_2)\in \Omega$ with $x_2>0$. This implies, in particular, that $u(x_1,x_2)\leq u(x_1,0)$, so it suffices to prove \eqref{c0est} along the $x_1$-axis.

We first claim that it is not possible that $u(x_1,0)>c^2/2$ on an interval $I$ of length larger than $2c$. 

Indeed, otherwise let $u_0$ be the radial solution to $\Delta u_0+2=0$, defined on a disk $D_c$ of radius $c$ and center $(a,b)$, with $u_0=0$ and $|Du_0|=c$ along $\parc D_c$. Then, $||u_0||_{L^{\8}(D_c)}= c^2/2$. Call $\cG_0$ to the graph of $u_0$ over $D_c$, and $\cG_+$ to the graph of $u$ over $\Omega\cap \{x_2\geq 0\}$. We choose the center $(a,b)$ so that $a$ is the midpoint of $I$ and $b<0$ is large enough so that $\cG_0$ and $\cG_+$ are disjoint. Then, we can slide $\cG_0$ horizontally in the $x_2$-direction towards $\cG_+$ until reaching a first contact point between both graphs. Such a contact point must be either an interior or a boundary point for both $\cG_0$ and $\cG_+$ simultaneously.  Hence we contradict the maximum principle in any case. This proves the claim.

It is obvious that \eqref{c0est} holds at the points where $u<c^2/2$. Now, let $x_1\in \R$ so that $u(x_1,0)>c^2/2$, and consider values $a_1<x_1<a_2$ such that $u(a_1,0)$ and $u(a_2,0)$ are both $\leq c^2/2$; that these two values exist comes from the previously proved claim. Let $U$ denote the bounded domain $U:= \Omega\cap \{ a_1<x_1<a_2\}$.

Next, define $$w(x_1,x_2):= u(x_1,x_2)-\frac{c^2}{2} - f_0^2 +x_2^2,$$ where $f_0>0$ is the supremum of $f$ in $\R$. Then $w$ is a harmonic function that has negative values along $\parc U$. By the maximum principle, $w\leq 0$ in $U$. Thus, $u(x_1,0)$ must satisfy \eqref{c0est}. This completes the proof.
\end{proof}

\begin{remark}\label{rem:rossic}
It is clear by the comparison principle that $\Omega$ cannot contain a disk of radius $R>c$. This implies that, for any $R>c$, every component of $\{f>R\}$ has length at most $2R$. Therefore, the components of $\Omega\cap\{x_2>R\}$ are bounded. Now, by Ros-Sicbaldi \cite{RS}, Lemma 6.1, this implies that ${\rm sup}_{\R} \,f \leq 4c$.  
\end{remark}

\begin{corollary}\label{cor:c1est}
There exists some $M=M(c)>0$ such that $|D u|\leq M$ in $\Omega$.
\end{corollary}
\begin{proof}
Since $u$ is bounded in $\Omega$ by Lemma \ref{prop:height}, the result follows directly from \cite{RSW2}, Lemma 2.2. Note that the bound $M$ only depends on $c$, not on $(u,\Omega)$.
\end{proof}

\subsection{Proof of Theorem \ref{th:estc2}} 
The proof is an adaptation of the curvature estimate of Ros, Ruiz and Sicbaldi, \cite[Prop. 3.1]{RRS1}. We first discuss some geometric facts we need to adapt their argument to our situation. The main difference is that in our case $\Omega$ has two boundary components. On the other hand, we can use the reflection symmetry of $\Omega$ to work only on the upper region $\Omega^+$ of the band. This will simplify the control of the blow-up limit. 

We start by writing
\[
\Omega^+=\{(x,y):0<y<f(x)\},\qquad
\sigma=\{(x,f(x)):x\in\mathbb R\},\qquad
L=\{y=0\}.
\]
By symmetry, the restriction of \(u\) to \(\Omega^+\) satisfies
\[
\Delta u=-2,\qquad
u=0,\quad \partial_\nu u=-c\quad\text{on }\sigma,
\qquad
\partial_yu=0\quad\text{on }L.
\]
We already know that \(u\) and \(Du\) are bounded; write \(|Du|\le M\).

\begin{lemma}\label{le35}
Let $p\in\sigma$ and $D_R(p)$ a disk centered at $p$. If $V$ is the component of $\Omega^+\cap D_R(p)$ adjacent to $p$, then
$
\partial V\cap\sigma\cap D_R(p)
$
is connected.
\end{lemma}
\begin{proof}
It is a straightforward adaptation of Lemmas 2.3 and 2.4 of \cite{RRS1}.
\end{proof}
The main step of the proof is:

\begin{claim}\label{claim1}
The geodesic curvature of $\sigma$ is bounded.
\end{claim}
\begin{proof}
Suppose $\sigma$ has unbounded curvature. Arguing as in the proof of \cite[Prop. 3.1]{RRS1}, we obtain points \(p_n\in\sigma\) and factors \(\lambda_n=|\kappa(p_n)|\to\infty\) such that a subsequence of the rescaled curves $
\sigma_n=\lambda_n(\sigma-p_n),$ which have $|\kappa_n(0)|=1$, 
converges \(C^{2,\alpha}\) on compact arclength intervals to a properly embedded curve $\Gamma_\infty$ satisfying
$$
0\in\Gamma_\infty,\qquad
|\kappa_\infty|\le1,\qquad
|\kappa_\infty(0)|=1.
$$
Define
$\Omega_n=\lambda_n(\Omega-p_n),$ and let $L_n$ denote the central symmetry line of $\Omega_n$, which is given by $L_n=\{x_2= d_n\}$ for some $d_n<0$. Consider also the associated functions 
\begin{equation}\label{defvnq}
u_n(z)=\lambda_nu\!\left(p_n+\frac z{\lambda_n}\right).
\end{equation}
Thus \(\Delta u_n=-2/\lambda_n\), \(|Du_n|\le M\), and
\begin{equation}\label{boccii}
\partial_\nu u_n=-c\quad\text{on }\sigma_n,
\qquad
\partial_{x_2}u_n=0\quad\text{on }L_n.
\end{equation}

If $d_n\to -\8$, we can argue as in steps 5-6 of  \cite[Prop. 3.1]{RRS1} and obtain a harmonic function $u_{\8}\in C^{2,\alfa}(\Omega_{\8})$ that satisfies $u_{\8}\geq 0$, $|\kappa_{\8}(0)|=1$ and the overdetermined boundary conditions
\begin{equation}\label{serrininf}
u_\8=0, \hspace{0.5cm} \frac{\parc u_{\8}}{\parc \nu} = -c \hspace{0.5cm} \text{ along $\parc \Omega_{\8}.$} 
\end{equation} By Traizet's theorem \cite{Tr}, $\Omega_{\8}$ is a half-plane, the complement of a disk, or a hairpin domain. The half-plane situation is impossible since $|\kappa_{\8}(0)|=1$. Also, $\Omega_{\8}$ is not the complement of a disk, since $\parc \Omega_\8=\Gamma_{\8}$ is unbounded. At last, $\Omega_{\8}$ cannot be a hairpin, since it only has one boundary component. 

Thus, we can assume up to subsequence that $d_n\to d_{\8}\leq 0$. Set $L_\infty:=\{x_2=d_\infty\}$. We claim that
\begin{equation}\label{nointere}
\Gamma_\infty\cap L_\infty=\varnothing.
\end{equation}
Indeed, suppose that $q$ belongs to this intersection. Since $\Gamma_\infty\subset\{x_2\ge d_\infty\}$, its tangent line at $q$ is horizontal.
Hence, after horizontally translating $q$ to $x_1=0$, it follows from our uniform geodesic curvature estimates and the convergence $(\sigma_n)_n\to \Gamma_{\8}$ that suitable arcs of $\sigma_n$ can be written as graphs $
x_2=d_n+h_n(x_1),$ $|x_1|<r_0,$
where
\begin{equation}\label{deribo}
h_n>0,\qquad h_n(0)\to0,\qquad h_n'(0)\to0,
\qquad |h_n''|\le C,
\end{equation}
with \(r_0,C\) independent of \(n\). 

For a fixed \(0<r<r_0\), apply the divergence theorem to $\nabla u_n$ on the domain
\begin{equation}\label{anr}
A_n(r)=\{(x_1,x_2):|x_1|<r,\ d_n<x_2<d_n+h_n(x_1)\}.
\end{equation}
By \eqref{boccii}, the flux through its lower side vanishes, while the exterior normal derivative on its upper side equals \(-c\). Consequently, if $I_1,I_2$ are the vertical sides of $A_n(r)$,
\[
c\int_{-r}^{r}\sqrt{1+h_n'(x_1)^2}\,dx_1
=
\frac{2}{\lambda_n}|A_n(r)|
+\int_{I_1\cup I_2}\partial_\nu u_n
\le
\frac{2}{\lambda_n}|A_n(r)|
+M\bigl(h_n(-r)+h_n(r)\bigr),
\] where we have used $|Du_n|\leq M$.
From \eqref{deribo} we have
$
h_n(-r)+h_n(r)\le 2h_n(0)+Cr^2.
$
Moreover, \(|A_n(r)|\) remains bounded. Letting \(n\to\infty\), we obtain
$
2cr\le MCr^2,
$
which is impossible for sufficiently small $r>0$. This contradiction proves \eqref{nointere}.

Since $\Gamma_\8$ is a properly embedded curve, it divides the plane into two connected components. In this way, by \eqref{nointere}, we can define $\Omega_\8$ as the region bounded by $\Gamma_{\8}$ and its reflection with respect to $L_\infty$. Then, arguing as in the case $d_n\to -\8$, we obtain a harmonic function $u_{\8}\in C^{2,\alfa}(\Omega_{\8})$ that satisfies $u_{\8}\geq 0$, $|\kappa_{\8}(0)|=1$ and the overdetermined boundary conditions \eqref{serrininf}. By Traizet's theorem \cite{Tr}, $\Omega_{\8}$ is a hairpin domain (it has two boundary components).

In order to rule out this blow-up behavior we will use the notion of capillary flux. Recall that, $\Omega_{\8}$ being a hairpin domain, it has non-zero capillary flux $\cF(\Omega_{\8})=(w,0)$, with $w\in \R$, $w\neq 0$; see \eqref{herflux}.

On the other hand, we can compute using \eqref{fluxsi} the capillary fluxes of the bands $\Omega_n$. Here we recall that the functions $u_n$ are defined by \eqref{defvnq} and satisfy $\Delta u_n=-2/\landa_n$, and that $\Omega_n=\landa_n (\Omega- p_n)$. So, from \eqref{fluxsi} we easily compute $$\cF(\Omega_n)=\landa_n\, \cF(\Omega),$$ where $\cF(\Omega)\in \R^2$ is the (finite) flux of our original Serrin band. (To be precise, we use \eqref{fluxsi} to compute the capillary flux of the translated band $\Omega_n+\landa_n p_n$, which  is symmetric with respect to the $x_1$-axis and has the same flux as $\Omega_n$).

Now, take compact transversal arcs $\gamma_n$ in the bands $\Omega_n$ that converge to a compact transversal arc $\gamma$ of the hairpin $\Omega_{\8}$. Then, the integrals in \eqref{defiflux} along the arcs $\gamma_n$ converge to the corresponding integral along $\gamma$. By the invariance of the capillary flux, this means that $\cF(\Omega_n)= \landa_n \,\cF(\Omega)$ must converge to $\cF(\Omega_{\8})$, which is a non-zero vector by \eqref{herflux}. This is impossible, since $\landa_n\to \8$. This contradiction rules out Case 2, and completes the proof of Claim \ref{claim1}.
\end{proof}
We now complete the proof of Theorem \ref{th:estc2}. We first show that the necks of a Serrin band do not collapse, i.e., $
\delta:=\inf_{\mathbb R}f$ is positive.

Arguing by contradiction, assume $\delta=0$, and choose $a_n\in\mathbb R$ such that $f(a_n)\to0$. Recall that $x_2=f(x_1)$ has bounded curvature, by Claim \ref{claim1}. Then, arguing as in \eqref{deribo}, there exist \(r_0,C>0\), independent of \(n\), such that
$
h_n(x_1):=f(a_n+x_1)$, $|x_1|<r_0$,
satisfies
$$
h_n>0,\qquad h_n(0)\to0,\qquad h_n'(0)\to0,
\qquad |h_n''|\le C.$$
Let $v_n(x_1,x_2):=u(a_n+x_1,x_2)$, and define in analogy with \eqref{anr} the domain 
\begin{equation}\label{anr2}
A_n(r)=\{(x_1,x_2):|x_1|<r,\ 0<x_2<h_n(x_1)\}.
\end{equation}
The same argument as the one following \eqref{anr} gives in our situation $$
2cr\le MC r^2+\frac{2C}{3}r^3,
$$
which is again impossible for sufficiently small $r>0$. Thus $\delta>0$.

Now, Claim \ref{claim1}, Lemma \ref{le35} and the graph estimates of \cite[Lemma 2.5]{RRS1} provide boundary charts of fixed radius $\varrho>0$, with uniformly bounded $C^{1,\alpha}$ norms. Since $\delta>0$, after choosing if necessary a smaller $\varrho>0$, each of these charts meets only one boundary component. Applying \cite[Lemma 2.1]{RRS1}, we obtain uniform $C^{2,\alpha}$ estimates up to the boundary. Interior elliptic estimates give the corresponding bounds away from the boundary. Consequently,
$\|u\|_{C^{2,\alpha}(\overline\Omega)}<\infty$ for every $\alpha\in(0,1)$, which completes the proof of Theorem \ref{th:estc2}.

\begin{remark}\label{rem:estibon}
In the proof of Theorem \ref{th:estc2} we have also established that, for any Serrin band $\Omega\subset \R^2$, it holds:
\begin{enumerate}
 \item
The distance between the two boundary components of $\parc \Omega$ is positive.
\item
$\parc \Omega$ is smooth (actually, real analytic), and has bounded geodesic curvature. Together with the previous boundary regularity estimates, this implies that $\parc \Omega$ has uniformly bounded $C^{k,\alfa}$ geometry on compact sets. 
\end{enumerate}
\end{remark}

From our previous arguments, we can easily establish a compactness result for the space of Serrin bands with fixed Neumann constant and non-collapsing necksizes.

\begin{corollary}\label{compaa}
Let $(\Omega_n)_n$ be a sequence of Serrin bands 
$$\Omega_n:=\{(x_1,x_2): |x_2|< f_n(x_1)\}\subset \R^2$$ with associated Neumann constant $c>0$, independent of $n$. Assume $f_n(x)\geq \delta>0$ for all $n\in \N$ and $x\in \R$, and some constant $\delta>0$. 
Then a subsequence of the $\Omega_n$ converges smoothly on compact sets to a Serrin band $\Omega_{\8}\subset \R^2$.

In particular, ${\rm sup} \{|f'(x)|: x\in \R\} <\8$.
\end{corollary}
\begin{proof}
The result is essentially a consequence of the proof of Theorem \ref{th:estc2}, so we just give an outline. Let $\Theta$ be the family of Serrin bands in the conditions of Corollary \ref{compaa}. Since $c>0$ is fixed, Lemma \ref{prop:height} and Corollary \ref{cor:c1est} give uniform $C^0$ and $C^1$ estimates across $\Theta$. Now, the proof of Theorem \ref{th:estc2} shows we also have uniform $C^2$ estimates, unless a blow-up subsequence of suitable dilated and translated domains $\hat{\Omega}_n=\landa_n (\Omega_n +p_n)$ converge to a hairpin domain $\Omega_0$. But this is ruled out by the condition $f_n\geq \delta>0$. 
Thus, we have uniform $C^2$-estimates across $\Theta$. Also, ${\rm sup} \{f_n(x): n\in \N, x\in \R\}$ is finite; see Remark \ref{rem:rossic}. From here, a standard Arzela-Ascoli argument of elliptic theory ensures a subsequence of the domains $\Omega_n$ converges to some Serrin domain $\Omega_{\8}$. By the uniform upper and lower bounds on $f_n$, we conclude $\Omega_{\8}$ has two boundary components, and so it is itself a Serrin band. This completes the proof of compactness. The fact that $\sup_\R |f'|$ is finite follows then easily, as follows. Assume there exists a sequence $(p_n)_n \subset \R$ with $|f'(p_n)|\to \8$. The translated domains $\Omega_n=\Omega- (p_n,0)$ subconverge to a Serrin band $\Omega_\8$, by compactness, that is still symmetric with respect to the $x_1$-axis. But at $x_1=0$, $\partial\Omega_\8$ has a vertical tangent line, in contradiction with the bigraph structure \eqref{bigraph}. 
\end{proof}

\section{Conformal geometry of Serrin bands}\label{sec:confo}

\subsection{Conformal parameters for Serrin bands}

Let $\Omega\subset \R^2$ be a Serrin band, and let $u(x_1,x_2)$ be its associated solution to $\Delta u+2=0$. By uniformization, there exists a bijective holomorphic map $g(z):\overline\cU\subset \C\flecha \overline\Omega$, where $\cU=\R\times (-y^*,y^*)\subset \C$ is a horizontal slab. Here, we identify $z=x+iy$ with $(x,y)\in \R^2$. We say that $z$ is a \emph{conformal parameter} for the Serrin band $\Omega$. Without loss of generality, we assume that $g(x,y^*)$ lies in the \emph{upper} boundary component of $\parc \Omega$. In this way, on a fixed horizontal band $\cU$, the map $g(z):\cU\flecha \Omega$ is unique up to horizontal translations $z\mapsto z+x_0$ in the parameter, with $x_0\in \R$.

By item (ii) of Remark \ref{rem:estibon} and standard boundary regularity theorems of complex analysis (see e.g. Theorem 3.6 in Pommerenke's book), we deduce that 
\begin{equation}\label{pome}
{\rm sup}_\cU \, |g^{(j)}| <\8, \hspace{0.5cm} \text{for every $j\geq 1$}.
\end{equation} 
Here we remark that Theorem 3.6 in Pommerenke is written for $\cU=\D$, the unit disk, but the proof works in our case due to our boundary estimates on the Serrin band $\Omega$ (item (ii) of Remark \ref{rem:estibon} above).

In this setting, we consider on $\overline\cU$ the harmonic function
\begin{equation}\label{def:ome}
\omega(x,y):=\log |g'(z)|,
\end{equation} which satisfies $2\omega_z = g''/g'$. The function
\begin{equation}\label{defuve}
v(x,y):=u(g(x+iy))
\end{equation} 
is a solution to
\begin{equation}\label{pdeom}
\Delta v +2 e^{2\omega}=0,
\end{equation} 
where $\Delta v= v_{xx}+v_{yy}$. From here,
\begin{equation}\label{def:q2}
\mathfrak{q}:=v_{zz}-2\omega_z v_z
\end{equation}
is holomorphic on $\overline\cU$. We call $\mathfrak{q}$ the \emph{Hopf differential} associated to $u$ in the conformal parameter $z=x+iy$. The Serrin overdetermined conditions imply that 
\begin{equation}\label{imcu}
{\rm Im} (\mathfrak{q})=0 \hspace{0.5cm} \text{along $\parc \cU$.}
\end{equation}
As a consequence of the curvature estimate in Theorem \ref{th:estc2}, we have:

\begin{proposition}\label{prop:qconst}
Let $\Omega$ be a Serrin band and $z$ a conformal parameter for $\Omega$. Then 
\begin{equation}\label{qcons}
\mathfrak{q} = {\rm const} \in \R.
\end{equation} 
\end{proposition}
\begin{proof}
Write $u=u(x_1,x_2)$ and $w=x_1+ix_2$. Then, from $v=u\circ g$ and $2\omega_z=g''/g'$ we have $$\mathfrak{q}=v_{zz}-\frac{g''}{g'} v_z = \left(u_{ww} \circ g\right) (g')^2.$$ So, by Theorem \ref{th:estc2} and equation \eqref{pome}, we deduce that $\mathfrak{q}$ is bounded in the conformal horizontal band $\cU$. Then, by \eqref{imcu}, Schwarz's reflection principle and Liouville's theorem, $\mathfrak{q}$ is constant.
\end{proof}

\begin{corollary}
For any $k\in \N$, we have $||v||_{C^{k}(\cU)}<\8.$
\end{corollary}
\begin{proof}
The case $k=0$ follows directly from Lemma \ref{prop:height}. The case $k\geq 1$ is an immediate consequence of \eqref{pome} and the fact that $||u||_{C^k(\Omega)}<\8$ by Theorem \ref{th:estc2}.
\end{proof}

\begin{corollary}\label{omebounded}
For any $k\in \N$, we have $||\omega||_{C^{k}(\cU)}<\8.$
\end{corollary}
\begin{proof}
By \eqref{pome}, it suffices to show that $|g'(z)|\geq c_0>0$ for some $c_0>0$. Consider $G:=g^{-1}:\Omega\flecha \cU$. Then, ${\rm Im}(G)$ is a bounded harmonic function in $\overline{\Omega}$ that is constant along each boundary component of $\Omega$. By Theorem \ref{th:estc2}, $\Omega$ has uniformly bounded geometry. So, by Schauder estimates, $|\nabla \,{\rm Im}(G)| = |G'|$ is bounded on $\Omega$. Hence, $|g'(z)|\geq c_0>0$, as desired.
\end{proof}
\begin{corollary}\label{kobe}
For any $k\in \N$ there exists $C=C(k,y^*)>0$ such that, if we denote $\eta:=\omega_z$, then
\begin{equation}\label{kbo}
|\eta_x^{(k)} (x,0)| \leq C, \hspace{0.5cm} \text{for all $x\in \R$}.
\end{equation}
\end{corollary}
\begin{proof}
Since $g(z)$ is holomorphic and injective in $\cU=\R\times(-y^*,y^*)$, one has the well-known estimate $$\left|\frac{g''(z)}{g'(z)}\right| \leq \frac{4}{{\rm dist}(z,\parc\cU)}.$$ In our case $\eta=g''/(2g')$, and so $|\eta|\leq 4/y^*$ on the smaller band $|y|\leq y^*/2$. So, by Cauchy estimates for $\eta(z)$ at any $z=x\in \R$, one has directly \eqref{kbo}.
\end{proof}

\subsection{Conformal representation of Serrin bands}\label{sec:confrep}

Let $\Omega\subset \R^2$ be a Serrin band, which is of the form \eqref{bigraph} without loss of generality, by \cite{RS}. Let $g(z):\R\times [-y^*,y^*]\flecha \overline\Omega$ be a conformal parametrization for $\Omega$, so that $g(x,0)$ parametrizes the $x_1$-axis, with $g'(x)>0$ for all $x\in \R$. By Proposition \ref{prop:qconst}, the Hopf differential $\mathfrak{q}$ associated to $g(z)$ is constant. 

As a matter of fact it can be easily proved that $\mathfrak{q}>0$. Specifically, \cite[Lemma 2.8]{CFM} gives, in our current $(x,y)$ coordinates, that $c\kappa =1-2\mathfrak{q}e^{-2\omega}$ along the boundary curve $y=y^*$. So, $\mathfrak{q}\leq 0$ would imply that the boundary curve of $\Omega$ is strictly convex, and this contradicts that its height over the $x_1$-axis is positive and bounded. So, after rescaling the complex parameter and the band $\cU$ we may chose $\mathfrak{q}=1/2$. Then, by Proposition 2.7 in \cite{CFM}, we have along $y=y^*$ that 
\begin{equation}\label{bofla}
\omega_y (x,y^*)= A \sinh \omega(x,y^*), \hspace{0.5cm} A:=\frac{2}{c}>0.
\end{equation}
By the symmetry of $\Omega$ as in \eqref{bigraph}, we have
\begin{equation}\label{simom}
\omega(x,-y)=\omega(x,y),
\end{equation} and so \eqref{bofla} holds along $y=-y^*$ for $-A$. Also, $\omega_y(x,0)=0$, i.e. $\omega_z(x,0)\in \R$. The associated solution $v(x,y)$ in \eqref{defuve} satisfies \eqref{pdeom} and
\begin{equation}\label{simv}
v(x,-y)=v(x,y).
\end{equation}
From \eqref{pdeom}, $v$ can be written in the form 
\begin{equation}\label{vh}
v(x,y):=h(x,y)-\frac{|g(z)|^2}{2}, \hspace{1cm} z=x+iy,
\end{equation} 
for some harmonic function $h(x,y)$ in $\R\times [-y^*,y^*]$. From \eqref{def:q2} and $\mathfrak{q}=1/2$, we have $v_{zz}-2\omega_z v_z=1/2$, and so 
\begin{equation}\label{edoach}
h_{zz}-2\omega_z h_z = \frac{1}{2}.
\end{equation} 
The general solution to \eqref{edoach} is, using that $2\omega_z=g''/g'$,
\begin{equation}\label{hh2}
h= {\rm Re}\left(\int^z g'(w) \varphi(w) dw\right), \hspace{0.5cm} \varphi'(z)=\frac{1}{g'(z)}.
\end{equation}
leaving the integration constants in \eqref{hh2} undefined for the moment.

The harmonic function $h(x,y)$ satisfies $h(x,-y)=h(x,y)$. Imposing this condition, we can write \eqref{hh2} as
\begin{equation}\label{hh}
h(x,y)=c_1 + c_2\, {\rm Re} g +{\rm Re}\left(\int_0^z g'(w) \phi(w) dw\right), \hspace{0.5cm} \phi(z)=\int_0^z\frac{1}{g'(w)}dw.
\end{equation} 
So, $h(x,y)$ is determined from $g$ by two real integration constants $c_1,c_2\in \R$. These constants represent the addition of linear terms to the torsion equation $\Delta u +2=0$. 

We now look at the Serrin boundary conditions $u=0$ and $\parc_\nu u=-c$ along the upper boundary curve of $\Omega$. For $v(x,y)$ as in \eqref{defuve}, these conditions are rewritten as 
\begin{equation}\label{sericon}
v(x,y^*)=0, \hspace{0.5cm} v_y(x,y^*)= -\frac{2}{A} e^{\omega}, \hspace{0.5cm} A:=\frac{2}{c}>0.
\end{equation} 
Let us write 
\begin{equation}\label{dete}
g'=e^{\omega}e^{i\vart}.
\end{equation} Then, $1/g'=e^{-\omega}e^{-i\vart}$. Therefore, from \eqref{vh}, \eqref{hh},  
\begin{equation}\label{seribo}
2v_z=g'(\varphi-\bar{g}).
\end{equation} 
for some function $\varphi$ with $\varphi'=1/g'$, and $\varphi(x)\in \R$ if $x\in \R$. Then, it follows from \eqref{sericon} and \eqref{seribo} that
\begin{equation}\label{ggam2}
g-\overline\varphi -\frac{2i}{A} e^{i\vart} =0 \hspace{0.5cm} \text{along $x\mapsto (x,y^*)$.} 
\end{equation}
In particular,
\begin{equation}\label{ggam4}
c_2:=g-\overline\phi -\frac{2i}{A} e^{i\vart} \in \R \hspace{0.5cm} \text{at $(0,y^*)$} .
\end{equation}

We discuss next the converse to this statement, in a slightly more general setting in which we allow ``bands'' $\Omega\subset \R^2$ with self-intersections, and we do not impose $v>0$.

To start, denote $\cU:=\R\times (-y^*,y^*)\subset \C$ for some $y^*>0$, and let $g(z):\overline\cU\flecha \C$ be a holomorphic mapping with $g'(z)\neq 0$ for all $z\in \cU$, and $g'(x)>0$ for all $x\in \R$. We also assume $g(x)\in \R$ if $x\in \R$. Then, we can define $\omega(x,y)\in C^{\8}(\overline\cU)$ as the associated harmonic function given by $|g'|^2 =e^{2\omega}$.  By Schwarz reflection, $\omega(x,-y)=\omega(x,y).$ 

We assume that \eqref{bofla} holds along the boundary component $\{y=y^*\}$ of $\cU$, and that \eqref{ggam4} holds at $(0,y^*)$. Note that $g$ need not be injective, and so $g(\cU)\subset \R^2$ is not, in general, a smooth domain.

We define next $v(x,y)$ by \eqref{vh}, \eqref{hh}. So in particular, $v:\overline\cU\flecha \R$ satisfies \eqref{pdeom}, 
and $\mathfrak{q}=v_{zz}-2\omega_z v_z=1/2$. And finally, since \eqref{ggam4} holds, the function $v(x,y)$ defined through the previous process satisfies Serrin boundary conditions \eqref{sericon}, after a suitable choice of the integration constants $c_1,c_2\in \R$. Specifically, writing $g'$ as in \eqref{dete}, Cauchy-Riemann gives $\vart_x=-\omega_y=-A\sinh\omega$ along $x\mapsto (x,y^*)$. From there $$\left(g-\overline\phi -\frac{2i}{A} e^{i\vart}\right)_x =0 \hspace{0.5cm} \text{along $x\mapsto (x,y^*)$.}$$ Using that \eqref{ggam4} holds for some real $c_2$, and setting $\varphi:=\phi+c_2$, equation \eqref{ggam2} holds and the rest of the process can be easily reversed.

{\bf To sum up:} starting from the holomorphic map $g(z):\overline\cU\flecha \C$ with $g'(z)\neq 0$, satisfying \eqref{bofla} and \eqref{ggam4}, we obtain a smooth solution $v$ to $\Delta v +2e^{2\omega}=0$ in $\overline\cU$, with Serrin overdetermined boundary conditions written in conformal parameters as \eqref{sericon}. The function $v(x,y)$ is given by \eqref{vh}, \eqref{hh} in terms of some specific constants $c_1,c_2\in \R$. So we can regard $\Omega:=g(\cU)$ as an immersed Serrin domain. 

\section{Serrin bands are algebro-geometric}\label{sec:kdv}

\subsection{Some preliminaries on mKdV theory}
Let us recall the mKdV operators, following the exposition in \cite{CFM,Jos}. They are the sequence of polynomial operators $P_n[\eta]$ acting on holomorphic functions $\eta=\eta(z)$ and its derivatives, and determined recursively by $P_{-1}[\eta]=0$ , $P_0[\eta]=\eta$ and
\begin{equation}\label{mkdv3}
P_{n+1}' -\frac{\eta'}{\eta}P_{n+1} = P_n''' -\frac{\eta'}{\eta} P_n'' -4\eta^2 P_n'.
\end{equation}
This differential equation has a solution up to an additive term $c_n \eta(z)$. In the case that we choose $c_j=0$ for all $j$, we obtain the   ``homogeneous'' mKdV operators $Q_n[\eta]$. The $Q_n[\eta]$ are homogeneous differential polynomials, see e.g. \cite{CFM}. One has 
\begin{equation}\label{relpq}
P_j =Q_j+ b_1 Q_{j-1}+\dots + b_j Q_0,
\end{equation}
with $b_1,\dots, b_j$ constants.

From \eqref{mkdv3} and $P_0=\eta$, we can easily compute the first few operators $Q_j$ as
$$Q_0 =\eta, \hspace{0.5cm} Q_1= \eta''-2\eta^3,\hspace{0.5cm} Q_2 = \eta^{(4)}-10 \eta'' \eta^2 -10 \eta (\eta')^2 +6 \eta^5.$$ We say that $\eta$ is an \emph{algebro-geometric potential} if for some $n\in \N$ it holds $$Q_{n}[\eta]= \beta_0 + \sum_{j=0}^{n-1} \alfa_j Q_j[\eta], \hspace{0.5cm} \beta_0, \alfa_0,\dots, \alfa_{n} \in \C.$$ This implies that all $Q_j[\eta]$ with $j>n$ are also linear combinations of the previous ones.

It will be important for our purposes to arrange these mKdV recursive polynomials into a Lax pair structure. This is standard. We follow here the development in Pinkall-Sterling \cite{PS}. A way to pass from $P_j$ to $P_{j+1}$ is by the system 
$$\left\{\def\arraystretch{1.5}\begin{array}{lll}P_{j+1}& =&  P_j''-2 \eta P_j' -2 \eta \tau_j,\\ \tau_j' &=& 2 \eta P_j' -P_j''\end{array}\right.$$ Note that there is an integration constant to be chosen for $\tau_j$ at each step $j$. We choose these constants to be real. 
We also consider functions $\sigma_j$ given by 
\begin{equation}\label{defisi}
\sigma_{j+1}:=\tau_j+2P_j'.
\end{equation}
We now define the formal series $$P=\sum_{n=-1}^{\8}  P_n \landa^n, \hspace{0.5cm}  \tau=\sum_{n=-1}^{\8} \tau_n \landa^n, \hspace{0.5cm} \sigma=\sum_{n=-1}^{\8} \sigma_n \landa^n,$$ where $P_{-1}=0$, $\sigma_{-1}=0$, and so $\tau_{-1}= -1/2$ and $P_0=\eta$. Then, 
\begin{equation}\label{dere}
P''-2\eta P' -2\eta \tau=\landa^{-1} P, \hspace{0.5cm} \tau'= 2\eta P'-P'', \hspace{0.5cm} \sigma=\landa \tau+2\landa P'.
\end{equation}
These equations are equivalent to \eqref{mkdv3}. Denote 
\begin{equation}\label{tildepot}
\xi_\landa:= \left(\def\arraystretch{1} \begin{array}{cc} P & \tau \\ \sigma & -P \end{array}\right), \hspace{0.5cm} \cW_\landa:= \left( \def\arraystretch{1} \begin{array}{cc} \eta & -\landa^{-1}/2 \\ -1/2 & -\eta \end{array}\right).
\end{equation}
Then, \eqref{dere} can be written compactly as
\begin{equation}\label{edotti}
 \xi'_{\landa} = [\xi_\landa,\cW_\landa].
 \end{equation}

\subsection{Algebro-geometric nature of Serrin bands}
By Proposition \ref{prop:qconst}, any Serrin band $\Omega\subset \R^2$ admits a conformal parametrization $g(z)$ from a horizontal band $\cU\equiv \R\times [-y^*,y^*]$, with respect to which $\mathfrak{q}=1/2$. Working in these parameters, we have:
\begin{theorem}\label{serrinfinite}
Any Serrin band is of finite type, i.e., in the above parameters $\eta:=\omega_z$ is an algebro-geometric mKdV potential.
\end{theorem}
\begin{proof}
From \cite{CFM}, all the hierarchic conformal Jacobi fields $h_j:={\rm Re}\, Q_j[\omega_z]$ of the mKdV hierarchy solve the Robin problem $\Delta h_j=0$ in $\R\times [0,y^*]$, with 
\begin{equation}\label{confojaco}
(h_j)_y=0 \hspace{0.5cm} \text{along $y=0$}, \hspace{1cm} (h_j)_y = A \cosh \omega\, h_j \hspace{0.5cm} \text{ along $y=y^*$.}
\end{equation} 
Moreover, by Corollary \ref{omebounded} we obtain directly that all $h_j$ are bounded. Now, from Lemma \ref{th:appendix} in the Appendix we deduce the existence of some $h_n$ that is a linear combination of the previous Jacobi fields $h_0,\dots, h_{n-1}$. This implies $\omega_z$ is an algebro-geometric mKdV potential, see \cite{CFM}.
\end{proof}

It is convenient for later use to reformulate algebraically the capillary boundary condition $\omega_y=A \sinh\omega$ along $y=y^*$ in \eqref{bofla} in terms of the hierarchic data $P,\sigma,\tau$. These type of capillary boundary conditions are known in the theory of integrable systems as \emph{integrable boundary conditions}; see e.g. \cite{AGGH,Kha,Sk} and specially \cite{KiS,Mc} for our situation.

\begin{proposition}\label{pro:releski}
For any $j\in \N$ it holds 
\begin{equation}\label{releski}
e^{\omega}\tau_{j-1} - e^{-\omega}\overline{\sigma_j} = \frac{2}{A} {\rm Im} P_j \hspace{0.5cm} \text{along $y=y^*$}.
\end{equation}
\end{proposition}
\begin{proof}
The case $j=0$ is immediate from $P_0=\omega_z$, $\tau_{-1}=\sigma_0=-1/2$ and $\omega_y=A \sinh\omega$. Now assume $j\geq 1$ and denote $P_{j-1}:=r+is$. Thus, $r_y=A r \cosh\omega$, and by Cauchy-Riemann,
\begin{equation}\label{fol1}
s_x =-A r \cosh \omega.
\end{equation}
From \eqref{dere} we have $\tau_{j-1}'= 2\eta P_{j-1}'-P_{j-1}''$. Using this and $2\eta=\omega_x-iA \sinh\omega$ we obtain
\begin{equation}\label{fol2}
{\rm Im}\,(\tau_{j-1})_x = \omega_x s_x - A r_x \sinh \omega - s_{xx}.
\end{equation}
We now put \eqref{fol1} and its derivative into \eqref{fol2} to obtain 
\begin{equation}\label{fol25}
\left({\rm Im}\, \tau_{j-1} - A e^{-\omega} r\right)_x =0.
\end{equation} 
Assuming the associated integration constant is zero (we will prove that below), we have
\begin{equation}\label{fol3}
{\rm Im} \, \tau_{j-1} = A e^{-\omega} r.
\end{equation}
Then, using \eqref{fol1}, \eqref{fol3} and the hierarchic relation \eqref{defisi} for $\sigma_j$, we arrive at 
\begin{equation}\label{fol9}
{\rm Im}\, \sigma_j = -A \,e^{\omega}\, {\rm Re}\, P_{j-1}.
\end{equation} 
Hence, 
\begin{equation}\label{fol4}
{\rm Im} \left(e^{\omega} \tau_{j-1} - e^{-\omega}\overline{\sigma_j}\right)=0.
\end{equation} 
The real part of the inner expression can be computed to be 
\begin{equation}\label{fol5}
2 \sinh\omega \,{\rm Re}\, \tau_{j-1} -2 e^{-\omega} r_x.
\end{equation} 
And finally, taking the imaginary part in $P_j=-\tau_{j-1}'-2\eta \tau_{j-1}$ and using \eqref{fol3}, we reach $${\rm Im} P_j= -A e^{-\omega} r_x +A \sinh\omega\, {\rm Re}\,\tau_{j-1}.$$ Putting this together with \eqref{fol4}, \eqref{fol5}, we obtain the desired expression \eqref{releski}.

We finally prove that the integration constant in \eqref{fol25} is indeed zero following a standard trick. A more formal derivation can be made following Section 10 in \cite{CFM}. First, write the expression $$\Phi={\rm Im}\, \tau_{j-1} - A e^{-\omega} r$$ obtained by integration in \eqref{fol25} as a differential expression in $\beta(x):=\omega(x,y^*)$ and its $x$-derivatives; this can be done using the harmonicity of $\omega$ and the capillary boundary condition $\omega_y=A \sinh\omega$ along $y=y^*$. Now, any real analytic $\beta(x)$ determines a harmonic function $\omega(x,y)$ by solving $\omega=\beta$, and $\omega_y=A \sinh\beta$ along $y=y^*$. Thus, the equation \eqref{fol25}, that is, $D_x\Phi=0$, can be seen as an identity of boundary jets. Comparing the highest one and then descending, we obtain that $\Phi$ is independent of all these jets. We can now compute $\Phi$ by taking the initial data $\beta$ to be constant. This gives a linear harmonic function for which $\Phi=0$. So, the integration constant is zero, as claimed. This completes the remaining detail of the proof.
\end{proof}

\subsection{The spectral curve of a Serrin band} 

With the previous notations, we will let $\Omega\subset \R^2$ be a Serrin band, and $z$ be a complex conformal parameter with respect to which $\mathfrak{q}=1/2$ holds. From Theorem \ref{serrinfinite}, $\omega_z$ is an mKdV potential, that is,
\begin{equation}\label{finitetype0}
Q_{n}[\omega_z]= i c_0 + \sum_{j=0}^{n-1} a_j Q_j[\omega_z], \hspace{0.5cm} c_0, a_j \in \R.
\end{equation} 
Here, $Q_j$ are the homogeneous mKdV recursive operators. Now, from the symmetry \eqref{simom} we obtain $\omega_z\in \R$ along $y=0$. Thus, evaluating \eqref{finitetype0} along $y=0$ gives $c_0=0$, that is, 
\begin{equation}\label{finitetype}
Q_{n}[\omega_z]= \sum_{j=0}^{n-1} a_j Q_j[\omega_z], \hspace{0.5cm} a_j \in \R.
\end{equation} 
Next, consider the non-homogeneous operators $P_j$ given by \eqref{relpq} for $b_r=-a_{n-r}$, where $a_0,\dots, a_{n-1}$ are the constants in \eqref{finitetype} and $1\leq r\leq n$. This means, by \eqref{finitetype}, that $$P_n=Q_n-\sum_{j=0}^{n-1} a_j Q_j =0.$$ Then $\tau_n'=0$, and we choose its associated integration constant so that $\tau_n=0$. This implies that the coefficients of $\xi_\landa$ vanish for $j\geq n+1$, i.e. $P_j=\tau_j=\sigma_j=0$ for all $j\geq n+1$, by taking all subsequence integration constants to be zero. Therefore, we have 
 \begin{equation}\label{def:pochiti}
\xi_{\landa} (z)=\left(\def\arrraystretch{1.5}\begin{array}{cc} 0 & -1/2 \\ 0 & 0 \end{array}\right) \landa^{-1} + 
\sum_{j=0}^n \left(\def\arrraystretch{1.5}\begin{array}{cc} P_j & \tau_j 
\\ \sigma_j & -P_j \end{array}\right) \landa^j, \end{equation} with 
\begin{equation}\label{pochi}
\left(\def\arrraystretch{1.5}\begin{array}{cc} P_n & \tau_n \\ \sigma_n & -P_n \end{array}\right)=\left(\def\arrraystretch{1.5}\begin{array}{cc} 0 & 0\\ \sigma_n & 0 \end{array}\right),
\end{equation} 
with $\sigma_n$ a holomorphic function. We will choose the smallest $n\in \N$ for which \eqref{finitetype} holds. In this way, we will assume without loss of generality that $\sigma_n(z)$ is non-zero, and $P_j=\tau_j=\sigma_j=0$ for $j>n$.

We can also check, using the highest order equation in \eqref{edotti} and $2\omega_z=g''/g'$, that 
\begin{equation}\label{sinc}
\sigma_n=c_0\, g', \hspace{0.5cm} \tau_{n-1}=\frac{c_1}{g'}, \hspace{0.3cm} c_0,c_1\in \R;
\end{equation}
Here, the reality condition $c_0,c_1\in \R$ come from the fact that $g'(x)>0$ along $y=0$. If we apply Proposition \ref{pro:releski} with $j=n$ and use $P_n=0$ we get 
\begin{equation}\label{sinc2}
e^{\omega}\tau_{n-1}=e^{-\omega}\overline{\sigma_n}.
\end{equation} 
Substituting \eqref{sinc} into \eqref{sinc2} and using $|g'|^2=e^{2\omega}$, we obtain 
\begin{equation}\label{sinc3}
c_1=c_0.
\end{equation}

\begin{remark}\label{xibounded}
By construction, all $P_j$ are functions of $\omega_z$ and its derivatives. In addition, by Corollary \ref{omebounded}, all $P_j$ and its derivatives are bounded on the band $\cU\subset \C$. Hence, $\xi_\landa$ is bounded on $\cU$.
\end{remark}

\begin{remark}\label{xisim}
By \eqref{simom}, we deduce that $P_j,\tau_j,\sigma_j$ take real values along $y=0$, for any $j\in \N$.
\end{remark}

Noting that $\xi_\landa$ takes values in $\mathfrak{sl}(2,\C)$, one can easily deduce that $\det \xi_\landa$ does not depend on $z$: 
\begin{equation}\label{nodet}
\parc_z \det \xi_\landa=-{\rm tr}\left(\xi_\landa \parc_z \xi_\landa\right)= -{\rm tr} \left(\xi_\landa \, [\xi_\landa,\cW_\landa]\right)=0.
\end{equation} 
We can then define the \emph{spectral polynomial} $a(\landa)$ as 
\begin{equation}\label{defa}
a(\landa):= \landa \det \xi_\landa.
\end{equation}
The polynomial $a(\landa)$ is of the form $$a(\landa)= -\frac{1}{4} + \cdots + a_{2n} \landa^{2n}.$$ 
By Remark \ref{xisim}, the coefficients of $a(\landa)$ are real. The leading factor of $a(\landa)$ is then 
\begin{equation}\label{sinc4}
a_{2n}=-c_0^2<0.
\end{equation}
We call $n\in \N$ the \emph{spectral genus} of the Serrin band $\Omega$.

\begin{lemma}\label{negrut}
On a Serrin band, all roots of the spectral polynomial $a(\landa)$ are negative.
\end{lemma}
\begin{proof}
Let $\landa_0\in \C$ be a root of $a(\landa)$. Note that $\landa_0\neq 0$, since $a(0)=-1/4$. Denote $\xi^0(x):=\xi_{\landa_0}(x,0)$. Then ${\rm trace}(\xi^0)={\rm det}(\xi^0)=0$, so we can write 
\begin{equation}\label{chifh}
\xi^0(x):= \left(\def\arrraystretch{1.5}\begin{array}{cc} -f h & f^2 \\ -h^2 & f h \end{array}\right) (x,0),
\end{equation}
for adequate smooth functions $f(x),h(x):\R\flecha \C$. By Remark \ref{xibounded}, both $f,h$ are bounded. Also, from \eqref{edotti} we have 
\begin{equation}\label{esro0}
f_x= -\eta f + \frac{1}{2\landa_0} h, \hspace{0.5cm} h_x= \frac{f}{2}+\eta h,
\end{equation}
where we note that $2\eta(x,0)=\omega_x(x,0)\in \R$ by the symmetry $\omega(x,y)=\omega(x,-y)$.

Arguing by contradiction, assume $\landa_0\not\in (-\8,0)$. Thus, there exists $\mu_0\in \C$ with ${\rm Re}(\mu_0)>0$ and $\mu_0^2=1/\landa_0$. Denoting $\Psi:= {\rm Re}(\mu_0 h \overline{f}):\R\flecha \R$, a computation from \eqref{esro0} gives 
\begin{equation}\label{desisi}
\Psi' =\frac{{\rm Re}(\mu_0)}{2} \left(|f|^2 + |\mu_0 h|^2 \right) \geq {\rm Re}(\mu_0)  |\Psi|.
\end{equation} 
If $\Psi$ is not identically zero, \eqref{desisi} would imply that $\Psi(x)$ increases at least exponentially as $x\to \8$ or $x\to -\8$; in particular, $\Psi$ would not be bounded, contradicting that $f,h$ are bounded. Thus, $\Psi(x)\equiv 0$. Again by \eqref{desisi}, this implies $f,h\equiv 0$, a contradiction with $\landa_0\neq 0$ by \eqref{chifh}. So, this proves that $\landa_0<0$, as claimed.
\end{proof}

\subsection{Serrin bands of low spectral genus}\label{sec:lowgenus}
We prove next that any Serrin band of spectral genus $n\leq 1$ is, up to similarity of $\R^2$, an element of the unduloidal family $\cF$. \begin{proposition}\label{prop:lowgenus}
Let $\Omega\subset\R^2$ be a Serrin band, and let $n\geq 0$ be the smallest integer for which \eqref{finitetype} holds. Then:
\begin{enumerate}
\item[(i)] $n=0$ if and only if $\Omega$ is a flat band.
\item[(ii)] $n=1$ if and only if $\Omega$ is, up to Euclidean similarity, one of the non-flat unduloidal bands $\Omega_\tau\in \cF$, $\tau\in(0,1)$.
\end{enumerate}
In particular, $n\leq 1$ if and only if $\Omega$ is, up to similarity, an element of $\cF$.
\end{proposition}

\begin{proof}
Assume $n=0$. Then, from \eqref{finitetype} and the convention $P_{-1}=0$ before \eqref{mkdv3} we have $Q_0[\omega_z]=\omega_z=0$. Since $2\omega_z=g''/g'$, we deduce that $g$ is affine, and so $\Omega$ is a flat band. Conversely, it is immediate that any flat band can be conformally parametrized by an affine developing map $g(z)$, which by the above equations has $n=0$.

All unduloidal Serrin bands $\Omega_\tau\in \cF$ have $n\leq 1$, see  \cite[Theorem 10.13]{CFM}. Conversely, assume next that $\Omega$ is a Serrin band with $n=1$. We use our usual conformal parameters for $\Omega$, with $\mathfrak{q}=1/2$. Then, by \cite[Theorem 10.13]{CFM}, there exist real analytic functions $\alfa(y),\beta(y)$ such that
\begin{equation}\label{lowg:foliation}
2\omega_y=-\alfa(y)e^{-\omega}-\beta(y)e^\omega.
\end{equation}
Here one should take into account that, in \cite{CFM}, the boundary curves of $\Omega$ are of the form $x={\rm const}$ rather than $y={\rm const}$, so one needs to interchange the roles of $x$ and $y$. By \cite[Proposition 3.3]{CFM}, $\alfa(y),\beta(y)$ satisfy
\begin{equation}\label{lowg:system}
\alfa''=\delta\alfa-2\alfa^2\beta,\qquad
\beta''=\delta\beta-2\alfa\beta^2,
\end{equation}
for some $\delta\in\R$. By symmetry, $\omega_y(x,0)=0$. In particular, the function $\omega(x,0)$ is not constant, since that would imply by Cauchy uniqueness that $\omega(x,y)$ is constant, and hence we would have $n=0$ as explained above. Therefore, from \eqref{lowg:foliation} we have $\alfa(0)=\beta(0)=0$. Now, by \eqref{lowg:system}, both $\alfa,\beta$ are solutions to the linear ODE $Y''=(\delta-2\alfa\beta)Y$. Since $\alfa(0)=\beta(0)=0$, we have that $\beta(y)$ is proportional to $\alfa(y)$. And then, comparing \eqref{lowg:foliation} with the Serrin condition \eqref{bofla} gives $\alfa(y^*)=-\beta(y^*)=A.$ Thus, $\beta(y)=-\alfa(y)$ and so 
\begin{equation}\label{lowg:planar}
\omega_y=\alfa(y)\sinh\omega,\qquad
\alfa''=\delta\alfa+2\alfa^3,\qquad \alfa(0)=0.
\end{equation}
Differentiating the first equation in \eqref{lowg:planar} in $y$, and using $\Delta \omega=0$, we obtain $$(\omega^2)_{xx} = 2(\omega_x)^2 -2 \alfa'(0) \,\omega \sinh\omega \hspace{0.5cm} \text{for all $(x,0)$}.$$ So, if $\alfa'(0)\leq 0$, then $\omega(x,0)$ is a convex function. Since $\omega$ is bounded, this would imply that $\omega(x,0)$ is constant, which we already ruled out.

Hence, $\alfa'(0)>0$. After a similarity of $\Omega$ and a rescaling of the conformal parameter, we can assume that $\alfa'(0)=1$. Now, equation (3.8) in \cite{CFM} gives
\begin{equation}\label{lowg:energy}
(\omega_x(x,0))^2=-2\cosh \omega(x,0)-\delta \hspace{0.5cm} \text{for all $(x,0)$.}
\end{equation}
Since $\omega(x,0)$ is non-constant, $\delta<-2$. Hence, there is a unique $\tau\in(0,1)$ such that
\begin{equation}\label{lowg:data}
\delta=-\left(\tau+\frac1\tau\right),\qquad
\alfa(0)=0,\qquad \alfa'(0)=1.
\end{equation}
These are the initial data in the construction of the unduloidal Serrin bands $\Omega_\tau$ in \cite{CFM}. Thus, $\Omega$ differs by a similarity of $\R^2$ from one of these bands. 

\end{proof}

\section{Uniqueness in the presence of a vertical symmetry}\label{sec:uniqueness}
The aim of this section is to prove the theorem below.
\begin{theorem}\label{mainth:sim}
Let $\Omega=\{(x_1,x_2):|x_2|\leq f(x_1)\}$ be a Serrin band. Assume $f(-s)=f(s)$. Then $\Omega$ is an element of the unduloidal family $\cF$.
\end{theorem}
We will divide the proof into two steps. We first describe our basic equations in terms of a vertical symmetry along the $x=0$ axis (Section \ref{sec:vs1}) and then prove the result in Section \ref{sec:genbon}.
\subsection{Equations in the presence of a vertical symmetry}\label{sec:vs1}
Let $\Omega$ be a Serrin band, parametrized in our usual conformal coordinates $(x,y)$, and assume it is symmetric with respect to $x=0$, i.e., $\omega(-x,y)=\omega(x,y)$. From Remark \ref{xisim} we have $P(x)\in \R$ for all $x\in \R$. Arguing similarly for the $x=0$ symmetry, we have $P(iy)\in i\R$ for all $y\in \R$.  In this way, if $\xi_\landa(z)$ is the holomorphic field in \eqref{def:pochiti}, and we let $z_0:=(0,0)$, then $P(z_0)=0$ and we have
\begin{equation}\label{zetzet}
\zeta_\landa(z_0)=  \left( \def\arraystretch{1.1}\begin{array}{cc} 0& T(\landa)/\landa \\ S(\landa)  & 0 \end{array}\right),
\end{equation}
where $T(\landa):=\landa \tau(z_0;\landa)$ and $S(\landa):=\sigma(z_0;\landa)$. The spectral polynomial $a(\landa)$ is given then by $$a(\landa)= \landa\, {\rm det}\, \xi_{\landa}(z_0) = -S(\landa)T(\landa).$$ In this way, $S(\landa),T(\landa)$ are real polynomials. 

We can now use \eqref{releski} together with $P(iy^*)\in i\R$ and \eqref{fol9} to deduce that
\begin{equation}\label{ecu6}
\landa e^{\omega} \tau - e^{-\omega} \sigma + \frac{2i}{A} P =0 \hspace{0.5cm} \text{at $z=iy^*$.}
\end{equation}
The left-side term on \eqref{ecu6} is a polynomial expression of degree $n$ in $\landa$. Using \eqref{sinc2} and the symmetry along $x=0$ gives $\sigma_n=e^{2\omega}\tau_{n-1}$ at $z=iy^*$. Now, $a(\landa)$ has degree $2n$ by \eqref{sinc4}, and so both $S(\landa),T(\landa)$ have degree $n$. By Lemma \ref{negrut}, all roots of $S(\landa), T(\landa)$ are negative.

Recall we are assuming without loss of generality that $n\in \N$ is the smallest number for which \eqref{finitetype} holds. We prove next that this implies that $S(\landa),T(\landa)$ have no common roots. Indeed, if $\landa_0$ is a common root of $S,T$, then $\xi_{\landa_0} (z_0)=0$, and from \eqref{edotti} we have $\xi_{\landa_0}(z)\equiv 0$. Then we can consider $$\hat{\xi}(z)=\frac{\xi_{\landa}(z)}{(1-\landa/\landa_0)}.$$ This gives a new polynomial solution to \eqref{edotti}, but with a lower degree. Reformulation of the mKdV recursive operators from $\hat{\xi}$ gives a contradiction with the minimality property of $n$. 

We next look at the Serrin condition. 
From $\mathfrak{q}=v_{zz}-2\omega_z v_z=1/2$ together with $\omega_x(0,y)=v_x(0,y)=0,$ we obtain along the symmetry line $x=0$ that 
\begin{equation}\label{inteparaderi}
\frac{d}{dy} \left(e^{-\omega(0,y)} v_y(0,y)\right)= -e^{\omega(0,y)}-e^{-\omega(0,y)}.
\end{equation} 
Note that $v_y(0,0)=0$. Hence we can integrate \eqref{inteparaderi} and evaluate at $y=y^*$ to obtain, using \eqref{sericon}, that 
\begin{equation}\label{sericon2}
\int_0^{y^*} \left(e^{\omega}+e^{-\omega}\right) dy = -e^{-\omega(0,y^*)} v_y(0,y^*) = \frac{2}{A}.
\end{equation}

\subsection{The spectral genus bound}\label{sec:genbon}

Let us now consider the holomorphic frame $F=F(z;\landa)$ obtained by solving the linear system
\begin{equation}\label{edofre}
F^{-1} F_z = \cW_\landa, \hspace{0.5cm} F(z_0)={\rm Id}.
\end{equation}
Note that ${\rm det}\, F=1$ since ${\rm tr}\, \cW_\landa=0$. Also, from \eqref{edotti}, \eqref{edofre} we obtain that $(F \xi_\landa F^{-1})_z=0$, and so, by \eqref{zetzet},  
\begin{equation}\label{zetzet2}
\xi_\landa(z)=F^{-1} \xi_\landa(z_0) F = F^{-1} \left( \def\arraystretch{1.1}\begin{array}{cc} 0& T(\landa)/\landa \\ S(\landa)  & 0 \end{array}\right) F.
\end{equation}
If we write $F=(F_{ij})$, then $F_{11}(iy)\in \R$ and $F_{12}(iy)\in i\R$, due to $\eta(iy)\in i\R$. So, it is convenient to denote $\psi_1(y,\landa):=F_{11}(iy)$, and $\psi_2(y,\landa):=-i F_{12}(iy)$. Thus, for any fixed real $\landa$, both $\psi_1,\psi_2$ are real functions. 

Let us next choose a root $\landa_0<0$ of $T(\landa)$. By $T(\landa_0)=0$ and \eqref{zetzet2} we have
\begin{equation}\label{rut1}
\xi_{\landa_0} (iy)= S(\landa_0) F^{-1}\left( \def\arraystretch{1.1}\begin{array}{rr} 0& 0 \\ 1 & 0 \end{array}\right) F = S(\landa_0) \left( \def\arraystretch{1.1}\begin{array}{cc} -i \psi_1 \psi_2 & \psi_2^2 \\ \psi_1^2  & i \psi_1 \psi_2 \end{array}\right).
\end{equation}
Then, putting together \eqref{ecu6} and \eqref{rut1} at $z=iy^*$, we obtain that the capillary boundary condition \eqref{ecu6} is written as 
\begin{equation}\label{rut2}
\cB(\landa_0)=0, \hspace{0.5cm} \cB(\landa):= \landa e^{\omega} \psi_2^2 - e^{-\omega} \psi_1^2 +\frac{2}{A} \psi_1\psi_2,
\end{equation}
where in the formula for $\cB(\landa):\R\flecha \R$ above we denote $\omega=\omega(iy^*)$ and $\psi_j=\psi_j(y^*,\landa)$. Note that both $\psi_1,\psi_2$ depend on $\landa$. Thus, any root of $T(\landa)$ is a negative zero of $\cB(\landa)$. 

Our next objective is to show that $\cB(\landa)$ has at most one negative zero. To start, it is convenient to write $\cB(\landa)$ alternatively as 
\begin{equation}\label{rut3}
\cB(\landa)= -\frac{\psi_1^2}{A} e^{-\omega} \left(A-2\Phi(\landa) - A \landa \Phi(\landa)^2\right), \hspace{0.5cm} \Phi(\landa):= e^{\omega(iy^*)} \, \frac{\psi_2(y^*,\landa)}{\psi_1(y^*,\landa)}.
\end{equation}
For $\Phi(\landa)$ to make sense, we prove next that $\psi_1(y^*,\landa)>0$ if $\landa<0$. First, from \eqref{edofre}, we find directly the system 
\begin{equation}\label{lasfis}
(\psi_1)_y = \frac{\omega_y}{2}\psi_1 + \frac{1}{2}\psi_2, \hspace{0.5cm} (\psi_2)_y= -\frac{1}{2\landa}\psi_1 - \frac{\omega_y}{2}\psi_2,
\end{equation}
with $\psi_1(0)=1$, $\psi_2(0)=0$. These equations imply that, for all $\landa<0$ and all $y\in (0,y^*]$, we have $\psi_1>0$, $\psi_2>0$. For instance, this can be seen by noting that $2(\psi_1 \psi_2)_y = \psi_2^2 -\psi_1^2/\landa$, which is non-negative if $\landa<0$. Thus $\Phi(\landa)$ is well defined for any $\landa<0$.

In order to control $\Phi(\landa)$, we define more generally
\begin{equation}\label{dipsi4}
\phi(y,\landa):=e^{\omega(0,y)}\, \frac{\psi_{2}}{\psi_{1}}(y,\landa):(0,y^*]\times \R^-\flecha \R^+.
\end{equation}
Observe that $\Phi(\landa)=\phi(y^*,\landa)$. Differentiating \eqref{dipsi4} using \eqref{lasfis} leads to 
\begin{equation}\label{dipsi}
\phi_y = -\frac{1}{2\landa} e^{\omega} - \frac{\phi^2}{2} e^{-\omega}, \hspace{0.5cm} \phi(0,\landa)=0.
\end{equation}
Now differentiate \eqref{dipsi} with respect to $\landa$, to obtain the linear differential equation in $y$ for $\phi_\landa:= \parc\phi/\parc \landa$ given by $$(\phi_\landa)_y + \phi e^{-\omega} \, \phi_\landa  = \frac{e^{\omega}}{2\landa^2}, \hspace{0.5cm} \phi_\landa (0,\landa)=0.$$ From here, it follows directly that $\phi_\landa(y,\landa)>0$ for $y\in (0,y^*]$. So, at $y=y^*$,
\begin{equation}\label{derisii}
\Phi'(\landa)>0, \hspace{0.5cm} \landa<0.
\end{equation}
On the other hand, for any $\landa<0$ we also have from \eqref{dipsi} $$0<\Phi(\landa) =\int_0^{y^*} \, \phi_y\, dy <-\frac{1}{2\landa} \int_0^{y^*} e^{\omega(0,y)} dy <-\frac{1}{2\landa}\, \frac{2}{A}= -\frac{1}{A\landa},$$ where the second inequality comes from the Serrin condition \eqref{sericon2}. Thus, 
\begin{equation}\label{derisii0}
1+A \landa \Phi(\landa)>0, \hspace{0.5cm} \landa<0.
\end{equation}

We finally bound the number of negative zeros of $\cB(\landa)$. Since $\psi_{1}>0$, we deduce that if some $\landa_0<0$ satisfies $\cB(\landa_0)=0$, then from \eqref{rut3}, \eqref{derisii}, \eqref{derisii0} we have $$\cB'(\landa_0)>0.$$
Thus, $\cB(\landa)$ has at most one negative zero, and such zero is simple if it exists.

Once here, we can complete the proof of Theorem \ref{mainth:sim}.

We have just shown that all zeros of $T(\landa)$ are negative zeros of $\cB(\landa)$, and there is only one of these. So $T(\landa)$ has at most one zero $\landa_0<0$. Let us now prove that $\landa_0$ is a simple root of $T(\landa)$. First, by linearity, write \eqref{zetzet2} as
\begin{equation}\label{zetzet11}
\xi_\landa(z)=
 \frac{T(\landa)}{\landa} F^{-1} \left( \def\arraystretch{1.1}\begin{array}{rr} 0& 1 \\ 0 & 0 \end{array}\right) F + S(\landa) F^{-1} \left( \def\arraystretch{1.1}\begin{array}{rr} 0& 0 \\ 1 & 0 \end{array}\right) F.
\end{equation} 
Now, taking $z=iy^*$, we can use the decomposition \eqref{zetzet11} in \eqref{ecu6} and obtain a relation $S(\landa) \cB(\landa)+T(\landa) \cC(\landa)=0$, where $\cB(\landa)$ is given by \eqref{rut1}, \eqref{rut2} as in the above proof. Here $\cC(\landa)$ is real analytic around $\landa_0$. Since $S(\landa_0)\neq 0$ and $\landa_0$ is a simple root for $\cB(\landa)$, this means $\landa_0$ is also a simple root for $T(\landa)$. Therefore, $T(\landa)$ has degree at most one, and since ${\rm deg}(T)=n$, we deduce $n\leq 1$. By Proposition \ref{prop:lowgenus}, $\Omega$ is a member of the unduloidal family $\cF$.

\section{The isospectral family}\label{sec:isospectral}

\subsection{Isospectral deformation theory}\label{sec:isodef}
Let $\Omega\subset \R^2$ be a Serrin band. In our usual parameters, take $\eta=\omega_z$. Consider for any meromorphic function $\phi=\phi(z)$ the higher-order meromorphic mKdV flows
\begin{equation}
\frac{\parc \phi}{\parc {\bf t}_j} :=-\frac{\parc}{\parc z} \, Q_j[\phi].
\end{equation}
In this way, each $\parc_{{\bf t}_j}$ is seen as a differential operator acting on $\phi$ and, by extension, on analytic expressions of its derivatives. 
One key property of these flows is that they commute: 
\begin{equation}\label{conmi}
\frac{\parc }{\parc {\bf t}_j} \left(\frac{\parc \phi}{\parc {\bf t}_k}\right)= \frac{\parc }{\parc {\bf t}_k} \left(\frac{\parc \phi}{\parc {\bf t}_j}\right).
\end{equation}

In Theorem \ref{serrinfinite} we proved that $\eta$ is an algebro-geometric mKdV potential, which satisfies \eqref{finitetype} for some real constants $a_0,\dots, a_{n-1}$. This means that the infinitesimal flow $\eta_{{\bf t}_n}$ is a linear combination of the lower order ones. Specifically, 
\begin{equation}\label{ag1}
\frac{\parc \eta}{\parc \mathfrak{s}}=0, \hspace{0.5cm} \frac{\parc \eta}{\parc \mathfrak{s}}:=\frac{\parc \eta}{\parc {\bf t}_n} - \sum_{j=0}^{n-1} a_j \frac{\parc \eta}{\parc {\bf t}_j},
\end{equation}
with $a_j\in \R$ given by \eqref{finitetype}. 

The fact that $\eta$ is an mKdV potential implies the global properties (1),(2) below. They follow from Section 2 of \cite{GW}, or alternatively from Meeks-Pérez-Ros \cite{MPR} after using the usual Miura transformation $u=\eta'-\eta^2$ to connect mKdV with the KdV hierarchy:
\begin{enumerate}
\item
$\eta(z)$, which was originally defined as a holomorphic function on the band $\cU=\R\times (-y^*,y^*)$, extends meromorphically to $\C$.
\item
Every pole $z_0$ of $\eta(z)$ is simple, and $$\eta(z)= \frac{m}{z-z_0} + O(z-z_0), \hspace{0.5cm} \text{for some $m\in \Z^*$}.$$
\end{enumerate}

By applying the Frobenius theorem, and due to the integrability condition \eqref{conmi}, we can argue as in \cite{MPR}, Proposition 5.3, and show that there exists a function $$\hat{\eta}=\hat{\eta}(z;\theta_0,\dots, \theta_{n-1}):\C\times \D^n(\ep)\flecha \overline{\C},$$ with each $\theta_j$ a complex variable, that satisfies $\hat{\eta}(\cdot, {\bf 0})= \eta$, $\hat{\eta}_{\bar{z}}=0$, and 
\begin{equation}\label{ecute}
\frac{\parc \hat{\eta}}{\parc \mathfrak{s}}=0, \hspace{0.5cm} \hat{\eta}_{\theta_j}=\frac{\parc \hat{\eta}}{\parc {\bf t}_j}, \hspace{0.5cm} j=0,\dots, n-1.
\end{equation}
So $\hat\eta(z;{\bf \theta})$ is a family of meromorphic functions of $z$, locally holomorphic in $(z,\theta)$ away from its poles. In principle, the argument is local, but the condition $\hat{\eta}_{\mathfrak{s}}=0$ implies $\hat\eta$ remains an mKdV potential, and so items (1)-(2) above apply globally.

In this sense, we obtain a holomorphic deformation $\hat{\eta}$ of $\eta$ by $n$ complex parameters $\theta_0,\dots, \theta_{n-1}$, so that for any fixed $(\theta_0, \dots, \theta_{n-1})$, the corresponding solution $\hat{\eta}(z)$ is an algebro-geometric mKdV potential with the same constants $a_0,\dots, a_{n-1}$ as those of $\eta$. That is, $\hat{\eta}$ also satisfies $$Q_n[\hat\eta]=\sum_{j=0}^{n-1} a_j Q_j[\hat\eta]$$ throughout the deformation. 

Now, one can define the holomorphic field $\hat\xi_\landa$ associated to $\hat\eta$ using \eqref{edotti}. It has an associated spectral polynomial $\hat{a}(\landa)=\landa\, {\rm det}\, \hat\xi_\landa$. Once here, it is well-known that mKdV hierarchy flows preserve the spectral polynomial; see e.g. Gesztesy-Ratnaseelan \cite{GR}, Section 4. Consequently, we have $\hat{a}(\landa)=a(\landa)$, independent of the deformation parameters. 
To ease the notation, we simply say that $\eta$ can be seen locally as a holomorphic function 
\begin{equation}\label{inteholo}
\eta=\eta(z,{\bf t}_0,\dots, {\bf t}_{n-1}),
\end{equation} 
where ${\bf t}_j\in \C$ denotes the flow along the $j$-th complex Jacobi field of the hierarchy, and the spectral polynomial $a(\landa)$ does not depend on ${\bf t}_0,\dots, {\bf t}_{n-1}$.

We now restrict to real times ${\bf s}_j={\bf t}_j\in \R$. Then, for $\omega$, we have the corresponding real flows
\begin{equation}\label{reflo}
\frac{\parc \omega}{\parc {\bf s}_j} =-2 {\rm Re}\, Q_j [\omega_z].
\end{equation} 
These flows also commute and let us consider 
\begin{equation}\label{oman}
\omega=\omega(z; {\bf s}_0,\dots, {\bf s}_{n-1}),
\end{equation} 
where the dependence on ${\bf s}_j\in \R$ is real analytic, and $\omega_{z\bar{z}}=0$. From properties (1)-(2) above, $\eta=\omega_z$ is a meromorphic function on $\C$ with simple poles of integer residues, and so $\omega$ is well defined in $\C$ minus these poles, at which it presents logarithmic singularities.

Note that, associated with each element $\omega$ as in \eqref{oman}, we can consider a holomorphic map 
\begin{equation}\label{isoge}
g=g(z;{\bf s}_0,\dots, {\bf s}_{n-1}),
\end{equation} 
defined at first on a simply connected neighborhood of the origin in which $\eta$ has no poles, and
given uniquely by the conditions 
\begin{equation}\label{inige}
g(0)=0,\hspace{0.5cm} g'(0)>0, \hspace{0.5cm} |g'|^2=e^{2\omega}.
\end{equation} Each $g'$ in this family is meromorphic in $\C$ with zeros and poles of even orders. Even though we will not use it, it can be shown that $g$ itself is also a meromorphic function in $\C$. Also, observe that $$\omega_{{\bf s}_0}=-\omega_x.$$ Thus, the action of ${\bf s}_0$ is given by translations $x\mapsto x+x_0$ of the horizontal conformal coordinate $x$. From the normalization $g(0)=0$ we are using, this ${\bf s}_0$ flow merely produces horizontal translations of the Serrin domain $\Omega$, by 
\begin{equation}\label{soflo}
\Omega_{{\bf s}_0}= \Omega-g(-{\bf s}_0).
\end{equation}

\begin{definition} Let $\Omega$ be a Serrin band, and let $g:\cU\flecha \Omega$ denote its normalized conformal parametrization. Associated to $\Omega$ we have then the algebro-geometric relation \eqref{finitetype} and the spectral polynomial $a(\landa)$.

 We define the \emph{isospectral family} $\mathcal I_\Omega$ of $\Omega$ as the set of all Serrin bands that have the same Neumann data $-c$, normalized conformal strip $\cU$, algebro-geometric relation \eqref{finitetype} and spectral polynomial $a(\landa)$ as $\Omega$. 
 
 By an abuse of notation we will also write $\omega\in \mathcal{I}_\Omega$ to denote that $\omega(x,y)$ is the harmonic function associated to a Serrin band of $\mathcal{I}_{\Omega}$.
\end{definition}

\subsection{Description of the isospectral family}
Let $\Omega$ be a Serrin band, with isospectral family $\mathcal{I}_\Omega$. Then we can consider the $n$-parameter family of harmonic functions $$\omega^{\bf s}=\omega(z; {\bf s}_0,\dots, {\bf s}_{n-1})$$ in \eqref{oman}, so that $\omega^{\bf 0} =\omega$. The next theorem shows that for ${\bf s}$ small enough, all these harmonic functions still generate Serrin bands in the isospectral family $\cI_{\Omega}$.

\begin{theorem}\label{isoteo}
There exists $\ep>0$ such that $\omega^{\bf s}\in \cI_\Omega$, whenever ${\bf s}\in (-\ep,\ep)^n$.
\end{theorem}
\begin{proof}
Let $\omega$ be the harmonic function $\omega(x,y)$ in $\cU$ associated to our Serrin band $\Omega$. Let $\xi_\landa(z)$ denote its associated holomorphic field \eqref{def:pochiti}, and $P_j$ be its corresponding non-homogeneous mKdV operators, related to $Q_j$ by \eqref{relpq}. 

To start we will fix any $j\in \{0,\dots, n-1\}$ and consider the infinitesimal hierarchic flow
\begin{equation}\label{jflow2}
\dot{\omega}= {\rm Re}\, P_j[\omega_z],
\end{equation}
where ``dot'' denotes the time-derivative at time zero along the corresponding local flow constructed in Section \ref{sec:isodef}. Note that $A,y^*$ remain fixed.

We will divide the proof into several claims. 

\begin{claim}\label{claim1b}
The infinitesimal flow \eqref{jflow2} preserves the capillary conditions
\begin{equation}\label{eskii0}
 \omega_y=0 \hspace{0.3cm} \text{along $y=0$}, \hspace{0.5cm} \omega_y = A\sinh\omega \hspace{0.3cm} \text{along $y=y^*$}.
 \end{equation} 
That is, it holds $(\dot\omega)_y=0$ along $y=0$, and $(\dot\omega)_y=A\cosh\omega\,\dot\omega$ along $y=y^*$.
\end{claim}
\begin{proof}
This is direct, since in \cite{CFM} it is proved that if a harmonic function $\omega(x,y)$ satisfying \eqref{eskii0}, then $h_j:={\rm Re} \, P_j[\omega_z]$ satisfies  
\begin{equation}\label{eskii1}
(h_j)_y=0 \hspace{0.3cm} \text{along $y=0$}, \hspace{0.5cm} (h_j)_y = A\cosh\omega \, h_j \hspace{0.3cm} \text{along $y=y^*$}.\end{equation}
\end{proof}
Next, recall that associated to $\omega(x,y)$ we have the holomorphic map $g$ in \eqref{isoge}. This produces via \eqref{vh}, \eqref{hh} a two-parameter family of solutions to $\Delta v + 2e^{2\omega} =0$ in $\cU$. In a Serrin band, one choice $(c_1,c_2)$ of the parameters within this family recovers the reparametrized solution $v=u\circ g$ in \eqref{defuve} for which the Serrin boundary condtions \eqref{sericon} hold. So it makes sense to define an infinitesimal deformation $\dot{v}$ of $v(x,y)$ in terms of $\dot{\omega}$ and $\dot{c_1}$, $\dot{c_2}$. With this, we have:

\begin{claim}\label{claim2}
The infinitesimal flow \eqref{jflow2} preserves the Serrin boundary conditions \eqref{sericon}. Specifically, there exist $\dot{c_1},\dot{c_2}\in \R$ such that the associated $\dot{v}(x,y)$ satisfies
\[
\dot v=0,\qquad
(\dot v)_y=-\frac2A e^\omega\, \dot\omega\,
\hspace{0.5cm} \text{along }y=y^*.
\]
\end{claim}

\begin{proof}
Write as in \eqref{dete},
$g'=e^{\omega}e^{i\vart}$, and let $\varphi$ be given by $\varphi'=1/g'$, and $\varphi(x)\in \R$ if $x\in \R$. In Section \ref{sec:confrep} we saw that the Serrin condition \eqref{sericon} is equivalent to \eqref{ggam2}, i.e., 
\begin{equation}\label{ggam20}
g-\overline\varphi -\frac{2i}{A} e^{i\vart} =0 \hspace{0.5cm} \text{along $x\mapsto (x,y^*)$,} 
\end{equation}
for some choice of the real integration constant that determines $\varphi(x)$

Now assume that \eqref{ggam20} holds, and consider the infinitesimal flow \eqref{jflow2}. Then 
\begin{equation}\label{gedo}
(\dot{g})'=g' P_j, \hspace{0.5cm} (\dot\varphi)' = -P_j/g'.
\end{equation} 
Now, from the recursion formulas
$$\tau_{j-1}'=-P_j-2\eta \tau_{j-1}, \hspace{0.5cm} \sigma_j'=P_j +2 \eta \sigma_j,$$ together with $2	\eta=g''/g'$, a computation shows that
\begin{equation}\label{dote1}
\dot{g}=-g' \tau_{j-1}+r_1, \hspace{0.5cm} \dot{\varphi}=-\frac{\sigma_j}{g'}+r_2,
\end{equation} 
with $r_1,r_2\in \R$. All these quantities are real along $y=0$. Fixing the choice $g(0)=0$ in the deformation, see \eqref{inige}, determines $r_1$, but the integration constant $r_2$ is in principle undetermined. We choose 
\begin{equation}\label{eres}
r_2=r_1
\end{equation}
for reasons that will become clear below. 
Denote $$\Gamma:= g-\overline\varphi-\frac{2i}{A} e^{i\vart}.$$ We have $\Gamma=0$ along $y=y^*$ by the Serrin condition \eqref{ggam20}. We now differentiate this equation using \eqref{dote1}, and obtain using also \eqref{eres} that
\begin{equation}\label{dote2}
\dot{\Gamma}=e^{i\vart}\left(-e^{\omega}\tau_{j-1}+e^{-\omega}\overline{\sigma_j}+\frac{2}{A} {\rm Im} (P_j)\right).
\end{equation}
The right-hand side vanishes by Proposition \ref{pro:releski}. Thus, $\dot\Gamma=0$ along $y=y^*$. Recall that, from \eqref{seribo},
$$
2v_z=g'(\varphi-\overline g)
=\frac{2i}{A}e^\omega-g'\, \overline\Gamma.$$
Differentiating this identity and using $\Gamma=\dot\Gamma=0$ along $y=y^*$, we obtain
$$
2(\dot v)_z=\frac{2i}{A}e^\omega\, \dot\omega.
$$
Taking real and imaginary parts yields
$$
(\dot v)_x=0,
\qquad
(\dot v)_y=-\frac2A e^\omega\, \dot\omega
\quad\text{along }y=y^*.
$$
In particular, \(\dot v\) is constant along this boundary. We choose the variation of the remaining real additive constant in \(v\) so that \(\dot v(0,y^*)=0\). Then \(\dot v=0\) along \(y=y^*\), proving the claim.
\end{proof}

We now consider the local hierarchy flow
\begin{equation}\label{jflow}
\parc_t \,\omega^t = {\rm Re}\, P_j[\omega_z^t], \hspace{0.5cm} \omega^0=\omega,
\end{equation}
which has \eqref{jflow2} as its infinitesimal version. That this flow can be integrated for $t$ small follows from Section \ref{sec:isodef}. Let $g^t$ denote 
the corresponding holomorphic maps, normalized by $g^t(0)=0$ and $(g^t)'(0)>0$ following \eqref{inige}. For sufficiently small $|t|$, these functions are defined on the fixed strip $\mathcal U$. Claim \ref{claim1b} implies that the capillary conditions \eqref{eskii0} are preserved along \eqref{jflow}. 

Now, regarding Serrin boundary conditions, we first note that the converse construction of Section \ref{sec:confrep} shows how to construct a solution $v^t(x,y)$ from $\omega^t$ and two integration constants $c_1^t, c_2^t$ via \eqref{vh}, \eqref{hh}. Then, Claim \ref{claim2} shows that we can choose for each $t$ in the deformation these constants $c_1^t,c_2^t$ t so that 
$$
v^t=0,\qquad
(v^t)_y=-\frac2A e^{\omega^t}
\hspace{0.5cm}\text{along }y=y^*.$$
By the symmetry $v^t(x,-y)=v^t(x,y)$, the corresponding Serrin boundary condition also holds along $y=-y^*$.

The corresponding conditions on the lower boundary follow by reflection. This implies that $\Omega^t:=g^t(\cU)$ is an \emph{immersed} Serrin band, as discussed in Section \ref{sec:confrep}. We next show that $\Omega^t$ is indeed a \emph{true} Serrin band for $t$ small.

\begin{claim}\label{claim3}
For $\ep>0$ small enough and $t\in (-\ep,\ep)$, all $\Omega^t:=g^t(\cU)$ are Serrin bands.
\end{claim}
\begin{proof}
To start, let us prove that for $t$ small, $g^t$ remains holomorphic in $\overline \cU$. Since $\eta^t$ is an algebro-geometric potential, it satisfies an algebraic holomorphic ODE, and so it is determined by its initial jet at an arbitrary point. By Corollary \ref{omebounded}, all the initial jets of $\eta^t=\omega^t_z$ are uniformly bounded on $\cU$. Then, we can apply the ODE argument of \cite[Prop. 5.3]{MPR} to our stationary mKdV system. This gives some $\ep>0$ such that $\eta^t$ is holomorphic on $\cU$ if $|t|<\ep$, with uniformly bounded jets. In particular, the associated map $g^t$ given by \eqref{inige} is also holomorphic in $\cU$ for $|t|<\ep$. Moreover, since we have by the computation in \eqref{gedo} that
$$
\partial_t(g^t)'=(g^t)'P_j[\eta^t],$$
then integrating in $t$ we get
$$
(g^t)'(z)
=
g'(z)\exp\!\left(\int_0^t P_j[\eta^s](z)\,ds\right).
$$
Using that $|g'|$ is bounded in $\cU$, this implies 
\begin{equation}\label{dosgp}
\|(g^t)'-g'\|_{L^\infty(\overline \cU)}
\le C\, |t|, \hspace{0.5cm} C>0.
\end{equation}

We next prove that, maybe for a smaller $\ep>0$, all $g^t$ with $|t|<\ep$ are injective in $\cU$. Let $\Omega\subset \R^2$ be our initial Serrin band, which is given by
$$\Omega=\{(x_1,x_2):|x_2|<f(x_1)\}.$$ Recall that the boundary components of $\parc\Omega$ are parametrized by $x\mapsto g(x\pm i y^*)$, and that $|g'|$ is bounded above and below by positive constants. Since $|f'|$ is also bounded on $\R$ (Corollary \ref{compaa}), this implies that ${\rm Re}\, g' \geq g_0>0$ on $\parc \cU$, and hence on $\cU$ by the maximum principle. From \eqref{dosgp}, choosing a smaller $\ep>0$, we can assume $${\rm Re}\,(g^t)'\geq g_0/2>0.$$
Now, by the mean value theorem, for any $z_1\neq z_2$ in $\cU$ we have
$$
\Re\!\left(
\frac{g^t(z_2)-g^t(z_1)}{z_2-z_1}
\right)
=
\int_0^1
\Re(g^t)'\bigl(z_1+s(z_2-z_1)\bigr)\,ds
\ge g_0/2>0.
$$
Therefore,
$$
|g^t(z_2)-g^t(z_1)|
\ge \frac{g_0}{2} |z_2-z_1|,
$$
so $g^t$ is injective and proper on the closed strip $\cU$. 

To see that $\Omega^t:=g^t(\cU)$ is indeed a Serrin band, it only remains to show that $v^t(x,y)$ is positive in $\cU$. Note that it suffices to prove that $|v_z^t|$ is bounded in $\cU$. Indeed, in that case, since $v^t=0$ along $\parc \cU$ by construction, we deduce by integration along vertical segments that $v^t$ is bounded. And once $v^t$ is bounded, using $v^t=0$ along $\parc \cU$ and $\Delta v_t=-2e^{2\omega^t}<0$, the maximum principle does the rest.

So, to prove that $|v_z^t|$ is bounded we recall \eqref{seribo}, which in our case gives $$2 v_z^t = (g^t)' \left(\varphi^t - \overline{g^t}\right).$$ Since $|(g^t)'|$ is bounded on $\cU$, we just need to bound $|\varphi^t - \overline{g^t}|$. But this holds directly since (i) along $y=y^*$ we have from $\Gamma=0$ that $|\varphi^t-\overline{g^t}|=2/|A|$, a constant, and (ii) we have $$(\varphi^t - \overline{g^t})_y= i\left(\frac1{(g^t)'}+\overline{(g^t)'}\right) \hspace{0.5cm} \text{in $\cU$}.$$ So $(\varphi^t - \overline{g^t})_y$ is bounded on $\cU$, due to our estimates for $(g^t)'$ and $1/(g^t)'$. Again, integrating along vertical segments proves $|\varphi^t - \overline{g^t}|$ is bounded in $\cU$, which was the remaining detail.

Hence, $\Omega_t$ is indeed a Serrin band for $|t|<\ep$ with $\ep>0$ small enough.
\end{proof}

Once here, the statement of Theorem \ref{isoteo} follows quite directly from Claim \ref{claim3}, as we explain next, First, the hierarchy flows \eqref{jflow} commute, and are linearly independent. This is due to the corresponding properties for the flows in the directions given by the homogeneous mKdV operators $Q_j$, and the relation \eqref{relpq} between them. So, any isospectral deformation ${\bf s}\flecha \omega^{\bf s}$ for ${\bf s}\in (-\ep,\ep)^n$ small enough is a composition of the flows \eqref{jflow} for $j=0,\dots, n-1$. Hence, all such harmonic functions $\omega^{\bf s}$, ${\bf s}\in (-\ep,\ep)^n$, generate Serrin bands so that $\omega^{\bf s}\in \cI_\Omega$. This completes the proof of Theorem \ref{isoteo}.
\end{proof}

\subsection{Compactness of the isospectral family}

\begin{theorem}\label{compaiso}
The isospectral family $\I:=\cI_{\Omega^0}$ of any Serrin band $\Omega^0\subset \R^2$ is compact.
\end{theorem}
\begin{proof}
We start by deriving a uniform lower bound on the \emph{necksize} of any Serrin band $\Omega\in \I$. First, recall that $\sigma_j, \tau_j$ are given by derivatives of $\eta=\omega_z$ up to order at most $2j+1$, and $\tau_j=\sigma_j=0$ for $j>n$. Then, by Corollary \ref{kobe} there exists some $C>0$, only depending on the original Serrin band $\Omega^0$ (and so independent of the chosen $\Omega\in \I$), such that $|\tau_{n-1}(x)|, |\sigma_n(x)| \leq C$.
From there, we can now use \eqref{sinc}, \eqref{sinc3} and \eqref{sinc4} to obtain 
\begin{equation}\label{dob}
\frac{\sqrt{-a_{2n}}}{C} \leq g'(x) \leq \frac{C}{\sqrt{-a_{2n}}}, \hspace{0.5cm} \text{for all $x\in \R$}.
\end{equation}
Now, for $x\in \R$ we can apply Koebe's $1/4$ theorem to the disk $D(x,y^*)\subset \cU$, which gives $$D\left(g(x), \frac{y^* g'(x)}{4}\right)\subset \Omega.$$ This condition together with \eqref{dob} gives
\begin{equation}\label{dob2}
{\rm dist}(g(x),\parc \Omega)\geq \frac{y^* \sqrt{-a_{2n}}}{4C}. 
\end{equation}
Noting that $g(\R)=\R$, we deduce from \eqref{dob2} that, if $\Omega=\{(x_1,x_2): |x_2|\leq f(x_1)\}$, then
\begin{equation}\label{dob3}
f(x)\geq \frac{ y^* \sqrt{-a_{2n}}}{4C} >0, \hspace{0.5cm} \text{for all $x\in \R$}.
\end{equation}
Once here, we can proceed to prove compactness. Let $\{\Omega_k\}_k$ be a sequence of Serrin bands in $\I$. By Corollary \ref{compaa} and \eqref{dob3}, a subsequence of the $\Omega_k$ converges to a Serrin band $\Omega_{\8}$ smoothly on compact sets. 
So, it remains to show that \(\Omega_\infty\in\mathbb I\). Let $g_k:\mathcal U\rightarrow\Omega_k$ be the conformal parametrization of the domain $\Omega_k$, normalized in our usual way: $g_k(0)=0$ and $g_k'(0)>0$. By Caratheodory kernel theorem (see e.g. Section 1.4 in \cite{Pom}), $g_k\rightarrow g_\infty$ locally uniformly on $\mathcal U$, where $g_\infty:\mathcal U\to\Omega_\infty$ is a conformal bijection. All their holomorphic derivatives converge locally uniformly as well, and so
\begin{equation}\label{convecon}
\omega_k=\log|g_k'|\rightarrow
\omega_\infty=\log|g_\infty'|,
\qquad
\eta_k=\frac{g_k''}{2g_k'}
\rightarrow
\eta_\infty=\frac{g_\infty''}{2g_\infty'},
\end{equation}
smoothly on compact subsets of $\mathcal U$. Now, the solutions $v_k:\cU\flecha \R$ associated to $\Omega_k$ are reconstructed from $g_k$ via \eqref{vh}, \eqref{hh}, and have Hopf differential $\mathfrak{q}_k=1/2$. Thus, at the limit, the corresponding Serrin solution $v_{\8}:\cU\flecha \R$ to $\Delta v_{\8}+2e^{2\omega_{\8}}=0$ associated to $\Omega_{\8}$ also has Hopf differential $\mathfrak{q}_{\8}=1/2.$

The relation \eqref{finitetype} holds for $\eta_k=(\omega_k)_z$, for all $k$, and so it passes to the limit. With similar arguments, one can define the limit holomorphic field $\xi_\landa^{\8}$ and easily show that  
$$\landa \det \xi_\landa^{\8}
=
\lim_{k\to\infty}\landa\det\xi_\landa^{(k)}
=a(\landa).$$
Hence,  $\Omega_\infty\in\mathbb I$, proving the desired compactness.
\end{proof}

\subsection{The hierarchy orbit} Given a Serrin band $\Omega$, we had defined a map $${\bf s}=({\bf s}_0,\dots, {\bf s}_{n-1})\in \R^n \mapsto \omega^{\bf s},$$ where $\omega^{\bf s}:=\omega(z;{\bf s})$ is the harmonic function in \eqref{oman}, and $\omega^{\bf 0}$ is the original harmonic function $\omega$ associated with $\Omega$. From Theorem \ref{isoteo}, there exists some $\ep>0$ such that $\omega^{\bf s}$ defines a Serrin band in the isospectral set $\cI_{\Omega}$ if ${\bf s}\in (-\ep,\ep)^n$. Now, using this fact and the compactness result of Theorem \ref{compaiso}, a standard ODE continuation argument shows that $\omega^{\bf s}\in \cI_{\Omega}$, for every ${\bf s}\in\R^n$.

\begin{definition}
The \emph{hierarchy orbit} $\cO_\Omega$ of a Serrin band $\Omega$ is defined as $$\cO_\Omega:=\{\omega^{\bf s} : {\bf s}\in \R^n\}\subset \cI_\Omega.$$ Here we identify as usual the Serrin band $\Omega$ with its associated harmonic function $\omega$.
\end{definition}
We will also use the natural closure of this orbit,
\begin{equation}\label{closo}
\mathcal K_{\Omega}:=\overline{\mathcal O_{\Omega}}.
\end{equation}
By Theorem \ref{compaiso}, $\mathcal K_{\Omega}\subset \mathcal{I}_{\Omega}$ is compact, and every $\Omega\in \cK_{\Omega}$ is a Serrin band.

\begin{remark}\label{mipre}
Recall that we are assuming that $\xi_\landa(z)$ in \eqref{def:pochiti} is the minimal holomorphic field of $\Omega$. It is easy to see that this minimality property transports to any element in $\cO_{\Omega}$.
\end{remark}

\section{Proof of Theorem \ref{th:main}}\label{sec:mainth}

\subsection{An isospectral symmetric Serrin band}
We prove next that, in the closure of an isospectral orbit $\cO_{\Omega^0}$ of any Serrin band $\Omega^0$ we can always find an element of the unduloidal family $\cF$.
\begin{theorem}\label{th:sime}
Given a Serrin band $\Omega^0$, there exists a Serrin band $\Omega\in \cK_{\Omega^0}$, of the form $\Omega=\{(x_1,x_2):|x_2|\leq f(x_1)\}$, that satisfies $f(-s)=f(s)$.

In particular, by Theorem \ref{mainth:sim}, $\Omega$ is an element of the unduloidal family $\cF$.
\end{theorem}
\begin{proof}
By Theorem \ref{compaiso}, the continuous function $\Omega\in \cK_{\Omega^0}\mapsto \omega_\Omega(0,0)\in \R$ attains its maximum at some Serrin band, that will be denoted simply by $\Omega\in \cK_{\Omega^0}$. Denote its usual associated data by $g,\omega,\eta=\omega_z$, and recall we are always assuming $g(0)=0$ and $g'(0)>0$.

Starting with $\Omega$, Theorem \ref{isoteo} provides an $n$-parameter deformation family of $\Omega$ through Serrin bands of $\cK_{\Omega^0}$. Differentiating the preceding functional along these variations gives 
\begin{equation}
\operatorname{Re}P_j[\eta](0)=0,
\qquad j=0,\ldots,n-1.
\end{equation}
Since \(\eta\) and all its derivatives are real along the real axis, we obtain
\[
P_j[\eta](0)=0,
\qquad j=0,\ldots,n-1.
\]
This means $\eta(0)=\eta''(0)=\cdots = \eta^{(2n-2)}(0)=0$. Now, we can view \eqref{finitetype} as an ODE $\eta^{(2n)}= \cP(\eta,\dots, \eta^{(2n-1)})$ with real coefficients. Since $-\eta(-z)$ satisfies the same ODE with the same initial conditions, we deduce that $\eta(-z)=-\eta(z)$. From here and $2\eta=g''/g'$ we get $g(-z)=-g(z)$, and since $g(x)\in \R$ for $x\in \R$ we conclude $g(-\bar{z})=-\overline{g(z)}.$ This proves the desired symmetry for $\Omega$.
\end{proof}

\begin{remark}
Assume that our initial Serrin band $\Omega^0$ in Theorem \ref{th:sime} is \emph{non-flat}, i.e., it is not a slab between two parallel lines. Then the symmetric band $\Omega$ is also non-flat, by the argument below.

We first show that $\omega(x,0):=\omega_{\Omega^0}(x,0)>0$ for some
$x\in\mathbb R$. Let $v(x,y)$ be its associated solution, see \eqref{vh}. Note that $\omega_y(x,0)=v_y(x,0)=0$, by symmetry. Then, by \eqref{pdeom} and $v_{zz}-(g''/g') v_z=\mathfrak{q}=1/2$, we have
$$v_{xx}(x,0)-\frac{g''(x)}{g'(x)} \, v_x(x,0) =1-e^{2\omega(x,0)}.$$ If $\omega(x,0)\leq0$ for every $x$, this means therefore that $\Psi(x):=v_x(x,0)/g'(x)$ satisfies $\Psi'(x)\geq 0$ for all $x$. Also, $\Psi(x)$ is not identically zero, since this would imply that $v(x,0)$ is constant, contradicting that $\Omega^0$ is non-flat.

Assume $\Psi(x^0)>0$ for some $x^0\in \R$. Then, for all $x\geq x^0$ we have $v_x(x,0)\geq \Psi(x_0) |g'(x)|\geq \delta>0$ for some $\delta>0$; here we recall that $|g'|\geq c_0>0$, see the proof of Corollary \ref{omebounded}. A similar argument can be done for $x\leq x^0$ if $\Psi(x^0)<0$. In any case, this is a contradiction with $v$ being bounded. Thus, the value $\omega_{\Omega}(0)\geq \omega(0)$ at the maximizer band is also positive, and this implies $\Omega$ is non-flat.
\end{remark}

\subsection{Isospectral isolation of non-flat unduloidal bands}

\begin{theorem}\label{th:isolated}
Let $\Omega^0$ be a Serrin band of minimal stationary order $n\geq 2$. Then $\mathcal K_{\Omega^0}$ cannot contain a non-flat unduloidal band.
\end{theorem}
\begin{proof}
Assume there exists a non-flat unduloidal band $\Omega\in \cK_{\Omega^0}=\overline{\cO_{\Omega^0}}$. So, there exists a sequence of Serrin bands $\Omega^{(k)}\in \cO_{\Omega^0}$, $k\in \N$, with associated holomorphic functions $\eta^{(k)}:=\omega_z^{(k)}$ and holomorphic Killing fields $\zeta_\landa^{(k)}$, converging as $k\to \8$ to their respective data $\eta$, $\zeta_\landa$ of $\Omega$. Without loss of generality, maybe after a horizontal translation, we place $\Omega$ so that $\Omega\cap\{x_1=0\}$ corresponds to a vertical segment of maximum length in $\Omega$, i.e., to a \emph{crest}.

Since $\Omega$ is unduloidal, its \emph{minimal} spectral genus is $n=1$, i.e. equation \eqref{finitetype} holds on $\Omega$ for $n=1$. That is, $Q_1[\eta]= r Q_0[\eta]=r \eta$ for some constant $r$. Recalling the notations in \eqref{tildepot} and \eqref{def:pochiti}, this implies that all $P_j[\eta]$ are also collinear with $\eta$, and we obtain 
\begin{equation}\label{dere00}P=P(z,\landa)=h(\landa) \eta(z),
\end{equation} 
for some polynomial $h(\landa)$ of degree $\leq n-1$. From \eqref{dere00} and \eqref{dere} we have
\begin{equation}\label{chicha}
\xi_\landa(z)= h(\landa) \hat\xi_\landa(z), \hspace{0.5cm} \hat\xi_\landa:=  \left(\def\arraystretch{1} \begin{array}{cc} \eta & \hat\tau \\ \hat\sigma & -\eta \end{array}\right).
\end{equation} 
Here $\hat\xi_\landa(z)$ denotes the \emph{minimal} holomorphic Killing field in $\Omega$, i.e., the one associated with the relation $Q_1[\eta]=r Q_0[\eta]$. This implies 
\begin{equation}\label{aaja}
a(\landa)=h(\landa)^2 \,\hat{a}(\landa),
\end{equation} 
where $\hat{a}(\landa)$ is the polynomial of degree two of associated with $\hat\xi_\landa$. In particular, $h(\landa)$ has degree $n-1$, and all its roots are negative (by Lemma \ref{negrut}).

The desired contradiction will come from comparing the two existing holomorphic fields $\xi_\landa(z)$ and $\hat\xi_\landa(z)$ at the unduloidal band $\Omega$ with respect to the capillary boundary conditions given by \eqref{releski}. To introduce \eqref{releski} into our computation, 
we first write the Lax equation \eqref{edotti} along the vertical axis. 

Denote $\Phi_\landa(y):=\xi_\landa(iy)$, $W(y):=\cW_\landa(iy)$, and $'=d/dy$. Then \eqref{edotti} gives 
\begin{equation}\label{edoveti}
\Phi_\landa' = i [\Phi_\landa, W], \hspace{0.5cm} \Phi_\landa(0)= \xi^0,
\end{equation} 
with $\xi^0:=\xi_\landa (0)$. Note that, by the symmetry along $y=0$, $\xi^0$ is real, i.e., $\xi^0\in \mathfrak{sl}(2,\R)$ for each fixed $\landa\in \R$; see Remark \ref{xisim}. Define next, motivated by \eqref{releski}, 
\begin{equation}\label{achh}
\cH(\landa):=\landa e^{\omega} \tau - e^{-\omega}\overline\sigma -\frac{2}{A} {\rm Im}\, P, \hspace{0.5cm} \left(\def\arraystretch{1} \begin{array}{rr} P & \tau \\ \sigma & -P \end{array}\right):= \Phi_\landa(y^*), \hspace{0.2cm} \omega:=\omega(iy^*).
\end{equation} 
Since $\Omega$ satisfies the capillary boundary condition along $y=y^*$, we have $\cH(\landa)=0$. The same equations hold for each $\xi_\landa^{(k)}$, and for $\hat\xi_\landa$.

This allows us to define a map $\cH=\cH(\landa,\xi^0)$ that assigns to each $\landa<0$ and each $\xi^0\in\mathfrak{sl}(2,\R)$ the complex value $\cH(\landa)\in \C$ in \eqref{achh} obtained after integrating the linear system \eqref{edoveti} and evaluating the resulting solution at $y=y^*$. For each fixed $\landa<0$, the map 
\begin{equation}\label{defach}
\cH_\landa\equiv \cH(\landa,\xi^0):\mathfrak{sl}(2,\R)\equiv \R^3\flecha \C
\end{equation} 
is linear. The initial condition $\xi^0=\xi_\landa(0)$ lies in the kernel of $\cH_\landa$, for any $\landa$ since, as we discussed above, we have $\cH(\landa)=0$ in that case. 

Our contradiction will come from the following two facts on the unduloidal band $\Omega$, that we establish in the next two sections.

\begin{claim}\label{cl2}
If $h(\mu)=0$, then $\mu<0$ and ${\rm Ker} \,\cH_\mu$ has dimension $\geq 2$.
\end{claim}
\begin{claim}\label{cl1}
For any $\landa<0$, ${\rm Ker} \,\cH_\landa$ is $1$-dimensional.
\end{claim}

These two claims are contradicting, so after their proofs in Section \ref{sec:cl2} and \ref{sec:cl1} below, this completes the proof of Theorem \ref{th:isolated}.
\end{proof}

\subsection{Proof of Claim \ref{cl2}}\label{sec:cl2}
Arguing by contradiction, assume ${\rm dim}\, {\rm Ker} \,\cH_\mu =1.$ This is an open condition on the space of linear maps $\mathfrak{sl}(2,\R)\to \C$. Thus, we can take a fixed linear complement $
\ell:\mathfrak{sl}(2,\mathbb R)\flecha\mathbb R$ such that $\Psi:=(\cH_\landa,\ell):\mathfrak{sl}(2,\mathbb R)\flecha\mathbb \C \times \R$ is a linear isomorphism, for all $\landa\in I:= (\mu-\ep,\mu+\ep)$ and $\ep>0$ small enough. In particular, there exist unique $V=V(\landa)\in \mathfrak{sl}(2,\R)$ such that $V(\landa)\in {\rm Ker}(\cH_\landa)$ and $\ell(V(\landa))=1$.

In a similar way, for some maybe smaller $\ep>0$, it follows from the convergence $\Omega^{(k)}\to \Omega$ that for every $k\in \N$ large enough and $\landa\in I$ there exists unique $V_k(\landa)\in \mathfrak{sl}(2,\R)$ satisfying
\begin{equation}\label{isold0}
\cH^{(k)}_{\lambda}\bigl(V_k(\lambda)\bigr)=0,
\hspace{0.5cm}
\ell^{(k)}\bigl(V_k(\lambda)\bigr)=1.
\end{equation} 
Moreover, $\cH_\landa^{(k)}\to \cH_\landa$ and $V_k(\landa)\to V(\landa)$ smoothly in I as $k\to \8$.

Now, for $\landa \in I$, the initial data $\xi_\lambda^{(k)}(0)$ belong to the kernel of $\cH_\landa^{(k)}$, since they satisfy \eqref{releski}. Since ${\rm Ker}\, \cH_\landa^{(k)}$ is $1$-dimensional, we have by \eqref{isold0}
\begin{equation}\label{isold1}
\xi_\lambda^{(k)}(0)=f_k(\lambda)V_k(\lambda),
\hspace{0.5cm}
f_k(\lambda):=\ell\bigl(\xi_\lambda^{(k)}(0)\bigr).
\end{equation}
Also, let us observe that $\xi_\mu^{(k)}(0)\neq 0$ since $\xi_\mu^{(k)}$ is associated to an element of the isospectral orbit $\cO_{\Omega^0}$, and $\xi_\mu(0)\neq 0$ at $\Omega^0$ by minimality hypothesis; see Remark \ref{mipre}. Thus, 
\begin{equation}\label{fkno}
f_k(\mu)\neq 0.
\end{equation} 
Passing  \eqref{isold1} to the limit gives 
\begin{equation}\label{isold2}
\xi_\lambda(0)=f(\lambda)V(\lambda), \hspace{0.5cm} f(\lambda):=\ell\bigl(\xi_\lambda(0)\bigr).
\end{equation} 
Thus, since the spectral polynomial $a(\landa)$ is preserved throughout $\cO_{\Omega^0}$, we have from \eqref{defa}, \eqref{isold1}, \eqref{isold2} that
\begin{equation}\label{mismaa}
\frac{a(\lambda)}{\lambda}
=f_k(\lambda)^2\det V_k(\lambda)
=f(\lambda)^2\det V(\lambda).
\end{equation}
By our hypothesis $h(\mu)=0$ and \eqref{aaja}, $a(\landa)$ has a zero of some order $m\geq 1$ at $\landa=\mu$, and in particular $\mu<0$. From \eqref{mismaa} and \eqref{fkno}, we see that $\det V_k$ also has a zero of the same order $m$ at $\landa=\mu$. Passing to the limit, $\det V(\landa)$ has a zero of order $\geq m$ at $\landa=\mu$. In particular, this implies $f(\mu)\neq 0$, which is a contradiction, since $\xi_\mu (0)=0$ by $h(\mu)=0$ and \eqref{chicha}. This proves Claim \ref{cl2}.

\subsection{Proof of Claim \ref{cl1}}\label{sec:cl1}
Let us first of all compute the map $\cH_\landa=\cH(\landa,\xi^0)$ of \eqref{defach} in coordinates, for any fixed $\landa<0$. Write 
$$
\xi^0=\begin{pmatrix}p&q\\r&-p\end{pmatrix}= p (E_{11}-E_{22})+ q E_{12} + rE_{21}, \hspace{0.5cm} p,q,r\in \R.$$ To start, we consider the solution $\Phi^1(y)$ to the vertical Lax system \eqref{edoveti}, given by 
\begin{equation}\label{edoveti2}
(\Phi^1)' = i [\Phi^1, W], \hspace{0.5cm} \Phi^1(0)=\begin{pmatrix}1&0\\0&-1\end{pmatrix}.
\end{equation} 
Recall that along the vertical symmetry line $x=0$ we have $\omega_x=0$ and $2\eta(iy)=-i\omega_y$. From this structure, the solution to \eqref{edoveti2} can be written as $$\Phi^1(y)= \begin{pmatrix}d(y)&i b(y)\\i c(y)&-d(y)\end{pmatrix},$$ where $b(y),c(y),d(y)$ are real smooth functions that satisfy, denoting $\omega:=\omega(0,y)$,
\begin{equation}\label{edoveti3}
d'=\frac{b}{2}-\frac{c}{2\landa},\hspace{0.5cm} (e^{\omega} b)'= -\frac{e^{\omega}}{\landa} d,\hspace{0.5cm} (e^{-\omega} c)'=e^{-\omega}\, d,\end{equation}
with $d(0)=1$ and $b(0)=c(0)=0$. Since $\lambda<0$, \eqref{edoveti3} and these initial conditions imply
\begin{equation}\label{edoveti4}
d(y)\ge1,\hspace{0.5cm} b(y),c(y)>0, \hspace{0.5cm} \text{for all $y>0$};
\end{equation}
specifically, as long as $d>0$, the last two equations make $b,c$ positive, and then the first makes $d$ increasing, since $\landa<0$. Hence $d\geq 1$ and \eqref{edoveti4} follows. We now compute $\cH_\landa(\Phi^1(0))=\cH_\landa(E_{11}-E_{22})$ using \eqref{achh}. First, for every $y$,
$$\landa e^{\omega} \tau - e^{-\omega}\overline\sigma -\frac{2}{A} {\rm Im}\, P= i\left(\landa e^{\omega}b+ e^{-\omega} c\right):=i \cJ(y).$$ Evaluating at $y=y^*$, we have $\cH_\landa(\Phi^1(0))=i \cJ(y^*)$. Now, a computation from \eqref{edoveti3} shows that 
\begin{equation}\label{edoveti11}
\cJ'(y)=-(e^{\omega}-e^{-\omega})d
\end{equation} 
for all $y>0$. Now, along any unduloidal band, if $x=0$ is the vertical segment corresponding to a crest (as we are assuming), one can check from the equations in \cite{CFM} (Lemma 9.1 and equation (9.3), taking into account that the $(x,y)$ coordinates here are exchanged with respect to \cite{CFM}) that $\omega(0,y)>0$ for every $y>0$. Thus, from \eqref{edoveti11} and since $d>0$, we conclude that $\cJ'(y)<0$  for all $y>0$. Since $\cJ(0)=0$, we deduce finally
\begin{equation}\label{jne}
\cJ(y^*)<0.
\end{equation} This gives a purely imaginary image vector for $\cH_\landa$:
\begin{equation}\label{edoveti10}
\cH_\landa(E_{11}-E_{22})\in i\R^-.
\end{equation}

We can do next a similar computation for
\begin{equation}\label{edoveti5}
(\Phi^2)' = i [\Phi^2, W], \hspace{0.5cm} \Phi^2(0)=\begin{pmatrix}0&q_0\\r_0&0\end{pmatrix}, \hspace{0.5cm} q_0,r_0\in \R.
\end{equation} 
This time the solution is of the form $$\Phi^2(y)= \begin{pmatrix}i p(y)& q(y)\\ r(y)&-ip(y)\end{pmatrix},$$ where $p(y),q(y),r(y)$ are real smooth functions. One obtains after some calculations, denoting $\omega:=\omega(0,y^*)$,
\begin{equation}\label{edoveti6}
\cH_\landa(\Phi^2(0))= \landa e^{\omega} q(y^*)-e^{-\omega} r(y^*) -\frac{2}{A} p(y^*).
\end{equation}
We claim that the quantity in \eqref{edoveti6} is not identically zero in the two-parameter space given by $q_0,r_0$ in \eqref{edoveti5} For that, we first note that
$$
\frac{d}{dy}\operatorname{tr}(\Phi^1(y)\Phi^2(y))=0,
$$
since both $\Phi^j$ solve the same Lax equation. At $y=0$, the trace is zero, so for every $y$,
\begin{equation}\label{edoveti7}
-i \operatorname{tr}(\Phi^1\,\Phi^2)=2dp+cq+br=0.
\end{equation}

From now on, let \(d,b,c,e^{\omega}\) denote their values at $y^*$, which we consider fixed. The values $(p,q,r)$ at $y^*$ of these solutions form exactly the plane \eqref{edoveti7}; indeed, since the vertical propagation along the Lax system is invertible, their space has dimension two, and the coefficient $2d$ is nonzero.
Suppose that \eqref{edoveti6} vanished for every initial $(q_0,r_0)$. Then the linear functional in \eqref{edoveti6} would vanish on the plane of $\R^3$ of implicit equation \eqref{edoveti7} in $(p,q,r)$. Their coefficient vectors would therefore be proportional:
$$
\left(-\frac2A,\lambda e^{\omega},-e^{-\omega}\right)
=s(2d,c,b)
$$
for some real $s$. In particular, $\lambda e^{\omega}=s\, c$, and $e^{-\omega}=-s\,b$.
These two identities imply $
\lambda e^\omega b+e^{-\omega}c=0,$ i.e., $\cJ(y^*)=0$, a contradiction with \eqref{jne}.
Hence, for some $(q_0,r_0)$, we have by \eqref{edoveti5}, \eqref{edoveti6} that $$\cH_\landa(q_0 E_{12}+r_0 E_{21}) \in \R\setminus \{0\}.$$ Together with \eqref{edoveti10}, this implies that the linear map $\cH_\landa:\mathfrak{sl}(2,\R)\flecha \C$ is surjective, and therefore has $1$-dimensional kernel. This concludes the proof.

\subsection{Completion of the proof of Theorem \ref{th:main}}
Let $\Omega^0\subset \R^2$ be a Serrin band different from a flat slab, and let $\cK_{\Omega^0}=\overline{\cO_{\Omega^0}}$ be the closure of its isospectral orbit. Then, within $\cK_{\Omega^0}$ there exists an unduloidal band $\Omega$ (Theorem \ref{th:sime}). This implies by Theorem \ref{th:isolated} that the minimal stationary order of $\Omega^0$ itself is $n\leq 1$. Thus, $\Omega^0$ is a member of the unduloidal family $\cF$, by Proposition \ref{prop:lowgenus}.

\section*{Appendix: An elliptic lemma}

Consider the band $\cU:=\R\times [0,L]\subset \R^2$, and the Robin boundary problem for harmonic functions $v(x,y)$ given by 
\begin{equation}\label{proble}
v_{xx}+v_{yy}=0 \text{ in $\cU$}, \hspace{0.5cm} v_y(x,0)=0, \hspace{0.5cm} v_y(x,L)= F(x) v(x,L),
\end{equation}
 for every $x\in \R$. Here we assume $F:\R\flecha \R$ is bounded.

\begin{lemma}\label{th:appendix}
In the above conditions, the space $\V$ of bounded solutions $v(x,y)$ to problem \eqref{proble} is finite-dimensional.
\end{lemma}
\begin{proof}
For a solution $v(x,y)\in \cV$, consider its cosine series expansion in $L^2(0,L)$,
$$v(x,y)=\sum_{n=0}^{\8} a_n(x) \phi_n(y), \hspace{0.5cm} a_n(x)=c_n\int_0^L v(x,y) \phi_n(y) \, dy,$$ where $k_n:=n\pi/L$ and $\phi_n(y):= \cos\bigl(k_n(L-y)\bigr)$. Note that $\phi_n(L)=1$. Since $\phi_n'(0)=\phi_n'(L)=0$, integration by parts and \eqref{proble} yields 
$$
\begin{aligned}
a_n''(x)
&=-c_n\int_0^L v_{yy}(x,y)\phi_n(y)\,dy\\
&=-c_n[v_y\phi_n]_0^L
+c_n[v\phi_n']_0^L
-c_n\int_0^L v\phi_n''\,dy\\
&=-c_ng(x)+k_n^2a_n(x),
\end{aligned}
$$
where $g(x):=F(x)v(x,L)$. That is,
\begin{equation}
a_n''-k_n^2a_n=-c_ng.
\label{eq:odes}
\end{equation}
Now, each \(a_n\) is bounded. Thus, for every $n\geq 1$, we have from \eqref{eq:odes} that
\begin{equation}
a_n(x)=c_n G_{k_n}*g = \frac{c_n}{2k_n}\int_{\R}e^{-k_n|x-t|}g(t)\,dt,
\label{eq:green}
\end{equation}
where $G_{k}(x):=e^{-k |x|}/(2k)$ is the Green kernel of $-\parc_{x}^2+k$. In particular,
\[
\lVert a_n\rVert_\infty\leq\frac{c_n}{k_n^2}\lVert g\rVert_\infty
\qquad(n\geq 1).
\]
Consequently, the cosine series converges uniformly, even at $y=L$, and we can define 
\begin{equation}\label{defw}
w(x):=v(x,L)=\sum_{n=0}^\8 a_n(x).
\end{equation}

The idea of the reminder of the proof is to show that $v(x,y)$ is determined by some finite number of the coefficient pairs $(a_n(0),a_n'(0))$. To do that, by linearity, it suffices to find $N=N(||F||_{\8},L)\in \N$ such that if
\begin{equation}
a_n(0)=a_n'(0)=0,\qquad 0\leq n\leq N,
\label{eq:data}
\end{equation}
then \(v\equiv 0\). Fix for the moment any \(N\geq 0\), assume \eqref{eq:data}, and take $\alfa\in (k_N,k_{N+1})$.

We first give an estimate for $a_n(x)$, with \(0\leq n\leq N\). Using the zero initial data \eqref{eq:data}, we can solve \eqref{eq:odes}, obtaining
for \(x\geq 0\),
\begin{equation}\label{eqan}
a_n(x)=-c_n\int_0^x\frac{\sinh\bigl(k_n(x-t)\bigr)}{k_n}g(t)\,dt,
\end{equation}
where the quotient inside the integral is understood as \(x-t\) when \(n=0\). Now, estimate $g(x)$ by $\cosh(\alfa x)$:
\begin{equation}
|g(x)|\leq  \lVert F\rVert_\infty\, M_\alpha\cosh(\alpha x)\label{eq:g-bound}, \hspace{0.5cm} M_\alpha:=\sup_{x\in\R}\frac{|w(x)|}{\cosh(\alpha x)} <\8,
\end{equation} where we recall $w(x)$ is bounded. And on the other hand, since \(\alpha>k_n\), one has
\[
\int_0^x\frac{\sinh\bigl(k_n(x-t)\bigr)}{k_n}\cosh(\alpha t)\,dt
=\frac{\cosh(\alpha x)-\cosh(k_nx)}{\alpha^2-k_n^2}.
\]
Then, putting \eqref{eq:g-bound} into \eqref{eqan}, and arguing in the same way for \(x\leq 0\), we obtain
\begin{equation}
|a_n(x)|\leq\frac{c_n\, \lVert F\rVert_\infty\, M_\alpha}{\alpha^2-k_n^2}\cosh(\alpha x),
\qquad 0\leq n\leq N.
\label{eq:low}
\end{equation}

We now estimate $a_n(x)$ for \(n\geq N+1\). This time we will use boundedness through formula \eqref{eq:green}, rather than
initial data. First, because \(k_n>\alpha\), we have
\[
\frac{1}{2k_n}\int_{\R}e^{-k_n|x-t|}\cosh(\alpha t)\,dt
=\frac{\cosh(\alpha x)}{k_n^2-\alpha^2}.
\]
Therefore, by \eqref{eq:green},
\begin{equation}
|a_n(x)|\leq\frac{c_n \, \lVert F\rVert_\infty\,  M_\alpha}{k_n^2-\alpha^2}\cosh(\alpha x),
\qquad n\geq N+1.
\label{eq:high}
\end{equation}
Adding \eqref{eq:low}, \eqref{eq:high}, and using \eqref{defw}, we obtain
$$
\frac{|w(x)|}{\cosh(\alpha x)}\leq \lVert F\rVert_\infty\, M_\alpha S(\alpha), \hspace{0.5cm} S(\alpha)=\frac{1}{L\alpha^2}
+\frac{2}{L}\sum_{n=1}^{\infty}\frac{1}{|k_n^2-\alpha^2|}.$$
Taking the supremum, by definition of $M_\alfa$, we arrive at
\begin{equation}
M_\alpha\leq \lVert F\rVert_\infty\, S(\alpha)M_\alpha.
\label{eq:contraction}
\end{equation}
Our objective is to prove the existence of some $\alfa$ such that $M_\alfa=0$, since that implies $w\equiv 0$, and from there $v=0$ follows easily. So, by \eqref{eq:contraction}, we must find $\alfa$ so that $\lVert F\rVert_\infty\, S(\alpha)<1$. For that, let us choose \(\alpha\) as the midpoint in $(k_N,k_{N+1})$, that is, 
\[
\alpha=\alpha_N=\frac{\pi}{L}\left(N+\frac{1}{2}\right).
\]
Writing \(q=N+\frac{1}{2}\), we have
\[
S(\alpha_N)=\frac{L}{\pi^2}
\left[\frac{1}{q^2}+2\sum_{n=1}^{\infty}\frac{1}{|n^2-q^2|}\right].
\]
This quantity tends to zero as \(N\to\infty\); specifically, by splitting the sum at \(n=N\) and using partial fractions, we obtain
\[
\frac{1}{q^2}+2\sum_{n=1}^{\infty}\frac{1}{|n^2-q^2|}
=\frac{2}{q}\sum_{j=0}^{2N}\frac{1}{j+\frac{1}{2}}.
\]
Hence
\begin{equation}
S(\alpha_N)
=\frac{2L}{\pi^2\left(N+\frac{1}{2}\right)}
\sum_{j=0}^{2N}\frac{1}{j+\frac{1}{2}}
\leq\frac{2L\bigl[2+\log(4N+1)\bigr]}{\pi^2\left(N+\frac{1}{2}\right)}
\longrightarrow 0.
\label{eq:small}
\end{equation}
Choosing \(N\) sufficiently large gives \(\lVert F\rVert_\infty\, S(\alpha_N)<1\). This implies $v\equiv 0$, as explained above, and proves the asserted finite-data uniqueness.

Therefore, $v(x,y)$ is determined by a finite-dimensional data space. Thus, the space $\cV$ is finite dimensional, as desired.
\end{proof}

\end{document}